\documentclass[dvips, 12pt, a4paper]{paper}
\usepackage{amsmath,amssymb,amsthm, dsfont}
\usepackage{color}
\definecolor{violet}{rgb}{0.5,0,0.8}

\theoremstyle{plain}
\newtheorem{theorem}{Theorem}[section]

\theoremstyle{plain}
\newtheorem{corollary}[theorem]{Corollary}

\theoremstyle{plain}
\newtheorem{lemma}[theorem]{Lemma}

\theoremstyle{plain}
\newtheorem{proposition}[theorem]{Proposition}

\theoremstyle{plain}
\newtheorem{definition}[theorem]{Definition}

\theoremstyle{plain}

\theoremstyle{plain}
\newtheorem{remark}[theorem]{Remark}

\begin{document}
\title{Martingale representation on enlarged filtrations: the role of the accessible jump times}
\author{Antonella Calzolari \thanks{Dipartimento di Matematica - Universit\`a di Roma
``Tor Vergata'', via della Ricerca Scientifica 1, I 00133 Roma,
Italy }  \and Barbara Torti $^*$}
\maketitle
\begin{abstract}
We consider a filtration $\mathbb{G}$ obtained as enlargement of a filtration $\mathbb{F}$ by a filtration $\mathbb{H}$.~We assume that all $\mathbb{F}$-local martingales are represented by a martingale $M$ and all $\mathbb{H}$-local martingales are represented by a martingale $N$.~$M$ and $N$ are not necessarily quasi-left continuous processes and their jump times may overlap.~We first analyze the contribution of the accessible jump times of $M$ and $N$ to the Jacod's dimension of the space of the $\mathcal{H}^1(\mathbb{G})$-martingales.~Then we prove a new martingale representation theorem on $\mathbb{G}$.
\end{abstract}

\begin{keywords}
Predictable representations property;
enlargement of filtration;
marked point processes;
compensators.
\end{keywords}

\section{Introduction}
Martingale representation is a classical topic of stochastic analysis with relevant theoretical and practical applications.~As well known, the martingale representation property is crucial to prove existence and uniqueness of BSDE's solutions, to derive the stochastic filtering equation by means of the innovation approach and to compute options hedging strategies in financial markets.~It is interesting to study what happens concerning martingale representation when the available information grows and typically the class of local martingales changes.~In this paper we investigate this problem assuming that the "representable" martingales may have jump times with non trivial \textit{accessible component}.~As far as we know, this assumption is an important novelty of our analysis compared to most of the papers on this subject.~We are  motivated by the growing interest of the literature in modeling situations of the real world which may
present critical "announced" random times, that is, from a mathematical point of view, \textit{predictable random times} (see e.g.~Fontana and Smith
 \cite{fontana-smith17}).~A second point of novelty is that we allow the \textit{additional information} to be not completely exogenous w.r.t.~to the \textit{basic information}.~In fact, rarely in the literature the two kinds of information, basic and additional, are produced by processes which may jump at the same random times (see e.g.~Jiao and Li \cite{jiao-li15}, Jeanblanc and Song \cite{jean-song15} and Aksamit, Jeanblanc and Rutkowski \cite{ak-jean-rut-19}).~We stress also that  only a few authors 
  deal, as we do here, with martingale representation when the growth of information is due to a process possibly different from the occurrence process of a random time $\tau$ (see e.g.~Kchia and Protter \cite{kchia-prott-15},  El Karoui, Jeanblanc and Jiao \cite{el-ka-jean-jiao15}, Di Tella and Jeanblanc \cite{ditella-jean-21}, Calzolari and Torti \cite{caltor22} and Biagini, Mazzon and Oberpriller \cite{bia-maz-ober-23}).~Moreover, the additional information is always generated by the observation of a process, while here we model the extra information as the reference filtration of a martingale, not necessarily its natural one\footnote{for an example of a martingale enjoying the PRP w.r.t.~a filtration larger than the natural one refer to the notion of  \textit{Weak Brownian Filtration} in \cite{mans-yor}}.~Based on the features highlighted in the above discussion,  our work meets the need suggested by the recent literature of new martingale representations on extended filtrations (see e.g.~Bandini, Confortola and Di Tella \cite{ba-con-ditella-21}  and  Bandini, Calvia and Colaneri \cite{band-calvia-col22}).
\bigskip\\
The framework of our study is the following.~On a given probability space $(\Omega,\mathcal{F}, P)$ let $M=(M_t)_{t\in [0,T]}$
and $N=(N_t)_{t\in [0,T]}$ be two square-integrable martingales enjoying the \textit{strong predictable representation property (PRP)} w.r.t.~the
filtrations $\mathbb{F}=(\mathcal{F}_t)_{t\in [0,T]}$ and
$\mathbb{H}=(\mathcal{H}_t)_{t\in [0,T]}$, respectively.~This means that all $(P,\mathbb{F})$ ($(P,\mathbb{H})$)-local martingales can be represented, up to the initial condition, as a stochastic integral w.r.t.~$M$ (w.r.t.~$N$).~Let $\mathbb{G}$ be the smallest right-continuous filtration containing both $\mathbb{F}$ and $\mathbb{H}$.~If there exists an equivalent probability measure such
that $\mathbb{F}$ and $\mathbb{H}$ are independent
(\textit{equivalent decoupling measure}), then it has been proved that the $\mathbb{R}^3$-valued process $(M, N, [M,N])$
enjoys the $\mathbb{G}$-PRP under the \textit{martingale preserving measure}, $P^*$, which is the decoupling measure under which $M$ and $N$ are
$\mathbb{G}$-martingales (see
\cite{caltor15}).~In other words any $(P^*,\mathbb{G})$-local martingale can be written as a \textit{vector stochastic integral} driven by  $(M, N, [M,N])$.~The fundamental argument is that the
covariation process $[M,N]$ is a $(P^*,\mathbb{G})$-square-integrable martingale strongly orthogonal to
$M$ and $N$.~So the set $(M, N, [M,N])$ turns out to be a
\textit{$(P^*,\mathbb{G})$-basis of martingales} in the sense of Davis and Varaiya (see \cite{davis1}).~The role of $\mathbb{F}$ and
$\mathbb{H}$ is symmetric and we can look at $\mathbb{G}$ either as the \textit{enlargement of
$\mathbb{F}$ by $\mathbb{H}$} or viceversa.\bigskip\\
In this paper we face two different issues linked to the above result.\bigskip\\
Our first problem starts from the following remark.~The process $[M,N]$, when not trivial, is a \textit{pure jump process} whose jump times arise by the overlapping of the accessible components of the jump times of $M$ and $N$.~Therefore $[M,N]$ vanishes almost surely as soon as at least one of the two martingales has \textit{totally inaccessible} jump times only.~The non-triviality of $[M,N]$ reflects on the dimension of the $(P^*,\mathbb{G})$-martingale driving the representation.~We investigate the consequences of this fact on the Jacod's dimension of the \textit{martingale space $\mathcal{H}^1(P,\mathbb{G})$}, that is the minimal possible dimension of a vector-valued local martingale enjoying the $(P,\mathbb{G})$-PRP.~We show that the maximum value of $dim\big(\mathcal{H}^1(P,\mathbb{G})\big)$ only depends on the mutual singularity of the kernels associated with
the sharp brackets of $M$ and $N$ joint with the mutual singularity of the kernels associated with the sharp brackets of their accessible martingale parts.~Based on the mutual behavior of the sharp brackets, we provide a condition equivalent to get $dim\big(\mathcal{H}^1(P,\mathbb{G})\big)$ equal to one and a condition sufficient to get $dim\big(\mathcal{H}^1(P,\mathbb{G})\big)$ equal to two.~\bigskip\\
The second problem we deal with is the construction of a local martingale driving the $(P,\mathbb{G})$-martingale representation.~More precisely, taking the point of view of the enlargement of $\mathbb{F}$ by $\mathbb{H}$,
we compute explicitly a multidimensional local martingale enjoying the $(P,\mathbb{G})$-PRP.~Indeed $N$ is a special $(P,\mathbb{G})$-semi-martingale and we denote by $N^\prime$ its local martingale part.~Our main hypothesis is that $P$ is the \textit{minimal martingale measure} for $N^\prime$ w.r.t.~$P^*$ (see \cite{ans_str92}).~We show that this condition implies that any $(P,\mathbb{F})$-martingale is a $(P,\mathbb{G})$-martingale (\textit{immersion of $\mathbb{F}$ in $\mathbb{G}$ under $P$})
and coincides with this property when $[M,N]\equiv 0$.~Under our assumptions we prove that $(M,N^\prime,[M,N])$ is a $(P,\mathbb{G})$-local martingale representing all the others.\bigskip\\
  The addressed issues lead us to prove two results of stochastic analysis of general interest.~More precisely we write the explicit expression of the compensator of a jump process with accessible jump times only and
   we show the existence of two predictable sets whose $\omega$-sections are disjoint supports of the random measures generated by two predictable increasing processes with mutually singular associated kernels.\bigskip\\
 The paper is organized as follows.~Section 2 is devoted to preliminaries.~After introducing the notations and recalling the necessary
 definitions and statements, we present in separated subsections some original results of the theory of semi-martingales
 interesting by themselves.~Section 3 describes the specific setting of our work.~In Section 4 we prove the theorem about the dimension of $\mathcal{H}^1(P,\mathbb{G})$.~The result on martingale representation is the object of Section 5.~First we discuss the case of progressive enlargement by $\tau$ and then the general case.~Finally in Section 6
 we formulate conclusions and future perspectives.
\section{Preliminaries}
Throughout the paper we consider processes on the time interval $[0,T]$.~Given a standard filtered probability space $(\Omega, \mathcal{A}, \mathbb{A}=(\mathcal{A}_t)_{t\leq T}, R)$ we denote by  $\mathcal{M}(R,\mathbb{A})$ the set of $(R,\mathbb{A})$-martingales ($\mathcal{M}_0(R,\mathbb{A})$ when the initial value is null) and by $\mathcal{M}_{loc}(R,\mathbb{A})$ ($\mathcal{M}_{loc,0}(R,\mathbb{A})$) the set of $(R,\mathbb{A})$-local martingales.~Analogously $\mathcal{M}^2(R,\mathbb{A})$ ($\mathcal{M}^2_0(R,\mathbb{A})$) is the set of square integrable $(R,\mathbb{A})$-martingales and  $\mathcal{M}^2_{loc}(R,\mathbb{A})$ ($\mathcal{M}^2_{loc,0}(R,\mathbb{A})$) its localization.\bigskip\\
\indent If $Y$ and $Y^\prime$ are two  semi-martingales we denote by $[Y,Y^\prime]$ their \textit{quadratic covariation process} and, when $Y=Y^\prime$, we  simply write $[Y]$ in place of $[Y,Y]$.~Given two local martingales $Z$ and $Z^\prime$ we denote by  $\langle Z,Z^\prime\rangle^{R,\mathbb{A}}$ their \textit{sharp bracket process} and, when $Z=Z^\prime$, we simply write $\langle Z\rangle^{R,\mathbb{A}}$ in place of $\langle Z,Z\rangle^{R,\mathbb{A}}$ (for definitions and existence's conditions we refer to Chapter VII in \cite{del-me-b}).\bigskip\\
\indent Given a semi-martingale $Y$ we denote by \textit{$\mathbb{P}(Y,\mathbb{A})$ the set of the local martingale
measures for $Y$ on $(\Omega, \mathcal{A}_T)$ equivalent to
$R|_{\mathcal{A}_T}$}.\bigskip\\
\indent Following Definition 4.22, page 121, in  \cite{he-wang-yan92}, a random time $\eta$ is a \textit{jump time} of a
c\`adl\`ag process $X$ if $R(X_{\eta}\neq X_{\eta^-}, \eta<+\infty)=R(\eta<+\infty)$.~As usual we denote with $\Delta X_{\eta}$ the
random variable $X_{\eta}- X_{\eta^-}$ with the convention that $\Delta X_{\eta}=0$ on $\{\eta = +\infty\}$.~When, as in this paper, $X$ lives
on the time interval $[0,T]$, its jump times are random variables taking values in $[0,T]\vee\{+\infty\}$.\bigskip\\
\indent A stopping time $\eta$ satisfies
 \begin{equation}\label{eq-eta-dec}
 \eta=\eta^{dp}\wedge\eta^{dq}
 \end{equation}
where $\eta^{dp}$ and $\eta^{dq}$ are the \textit{accessible component} and the \textit{totally inaccessible component}, respectively (see e.g.~Theorem 3,
 page 104, in \cite{Prott}).~More precisely
there exist two disjoint events $A$ and $B$ such that  $P$-a.s.
$$A\cup B=(\eta<\infty)$$
 and
 \begin{equation}\label{eq-time-decomposition}
   \eta^{dp}=\eta\, \mathbb{I}_A + \infty\, \mathbb{I}_{A^c},\;\;\;\eta^{dq}=\eta\, \mathbb{I}_B + \infty\, \mathbb{I}_{B^c}.
 \end{equation}
Therefore
 $$\eta \mathbb{I}_{\eta<+\infty}=\eta^{dp}\, \mathbb{I}_A + \eta^{dq}\, \mathbb{I}_B.$$
It is worthwhile to recall that
 \begin{equation}\label{eq: accessible-r-t}
 [[\eta^{dp}]]\subset\cup_m [[\eta^{dp}_m]],
 \end{equation}
where $\{\eta^{dp}_m\}_{m\in\mathbb{N}}$ is a sequence of predictable stopping times.~Such a sequence is not unique and it may be chosen in such a way that the corresponding graphs are pairwise disjoint (see Theorem 3.31, page 95, in \cite{he-wang-yan92}).~As far as $\eta^{dq}$ is concerned, we recall that by definition
 \begin{equation*}
 R(\eta^{dq}=\sigma<+\infty)=0,
 \end{equation*}
 for any predictable stopping time, and therefore for any accessible stopping time, $\sigma$.\bigskip\\
\textit{From now on we call  enveloping sequence of $\eta^{dp}$ any sequence
of predictable stopping times satisfying (\ref{eq: accessible-r-t}) and with pairwise disjoint graphs}.
\begin{remark}\label{rem:total-inacess}
We note that any finite totally inaccessible random time cannot coincide with positive probability with an independent random time.
\end{remark}
 Let us state Yoeurp's result about orthogonal decomposition of a local martingale.
\begin{theorem}\label{thm-decomp-yoeurp}(Theorem 1-4 in \cite{yoeurp76})\\
If $Z\in\mathcal{M}_{loc,0}(R,\mathbb{A})$ then $Z$
can be uniquely decomposed as
\begin{equation}\label{eq-decomp-yoeurp}
   Z=Z^c+Z^{dp}+Z^{dq},
\end{equation}
where $Z^c$, $Z^{dp}$ and $Z^{dq}$ belong to $\mathcal{M}_{loc,0}(R,\mathbb{A})$ and $Z^c$ has continuous trajectories, $Z^{dp}$ has
accessible jump times only and is strongly orthogonal to any local
martingale with at most totally inaccessible jump times, $Z^{dq}$
has totally inaccessible jump times only and is strongly
orthogonal to any local martingale with at most accessible jump
times.
\end{theorem}
\noindent As usual we refer to  $Z^c$, $Z^{dp}$ and $Z^{dq}$ as the
\textit{continuous}, the \textit{accessible martingale part} and the
\textit{totally
inaccessible martingale part} of $Z$, respectively.~We recall that $Z^{dp}$ and $Z^{dq}$ are  purely discontinuous local martingales
(see Definition 4.11, page 40, in \cite{ja-sh03}).~Observe that to any jump time of $Z^{dp}$ ($Z^{dq}$) corresponds the accessible
(totally inaccessible) component
of a jump time of $Z$ and viceversa.\\
It is worthwhile to stress that the above decomposition depends
on the reference filtration.~In fact  the nature
 of a random time is linked to the choice of the filtration and in particular \textit{accessibility is preserved by enlarging the
 filtration and viceversa total inaccessibility is preserved under restriction of the
 filtration}.\bigskip\\
\indent The crucial notion of this paper is the \textit{strong predictable representation property} of a local martingale.~We refer to \cite{cha-stri94} for the definition of the \textit{vector
stochastic integral} and its relation with the \textit{componentwise
stochastic integral}.~We recall that the two notions coincide when the components of the driving local martingale are pairwise strongly orthogonal (see Theorem 3.1 in \cite{cha-stri94}).
\begin{definition}(Definition 13.1, page 362, in \cite{he-wang-yan92})\\
$\bold{Z}=(Z_1,...,Z_m)$ with $Z_i\in M_{loc}(R,\mathbb{A}), i=1,\ldots, m,$ enjoys the $(R,\mathbb{A})$-strong predictable representation property ($(R,\mathbb{A})$-PRP) if each $V\in M_{loc,0}(R,\mathbb{A})$ can be represented  as vector stochastic integral w.r.t.~$\bold{Z}$.
\end{definition}
\noindent Next theorem is known as \textit{II Fundamental Theorem of Asset Pricing}.
\begin{theorem}\label{thm-2assetpricing}(Theorem 13.9, page 366, in \cite{he-wang-yan92})\\
 Let $\mathcal{A}_0$ be trivial.~Then $\boldsymbol{Z}\in\mathcal{M}_{loc}(R,\mathbb{A})$ enjoys the $(R,\mathbb{A})$-PRP if and only if $\mathbb{P}(\boldsymbol{Z},\mathbb{A})$ is a singleton.
\end{theorem}
\noindent By the above theorem and the comment to Corollary 11.4, page 340, in \cite{jacod} it follows that
 $\bold{Z}=(Z_1,...,Z_m)$, with $Z_i\in \mathcal{M}^2_{loc}(R,\mathbb{A}), i=1,\ldots, m$, enjoys the $(R,\mathbb{A})$-PRP if and only for each $V\in \mathcal{M}^2(R,\mathbb{A})$ holds
\begin{align*}
V_t=V_0+\left(\xi^V\cdot \bold{Z}\right)_t,
\end{align*}
where $V_0$ is $\mathcal{A}_0$-measurable,
$\xi^V=(\xi^V_1,...,\xi^V_m)$ is an $m$-dimensional
predictable  process
 such that
\begin{equation}\label{eq-integrands}
E^R\left[\sum_{i,j}\int_0^T\xi^V_i(t)\xi^V_j(t)\,d[Z_i,Z_j]_t\right]<+\infty\end{equation}
and
 $\left(\xi^V\cdot \bold{Z}\right)_t$ denotes the vector stochastic
integral of $\xi^V$.\\
The isometry between $\mathcal{M}^2(R,\mathbb{A})$ and $L^2(\Omega,
 \mathcal{A}_T, R)$ implies that, if  $\bold{Z}=(Z_1,...,Z_m)$, with $Z_i\in \mathcal{M}^2_{loc}(R,\mathbb{A}), i=1,\ldots, m$, enjoys the $(R,\mathbb{A})$-PRP, then
 each $K\in L^2(\Omega, \mathcal{A}_T, R)$ satisfies
$
K=K_0+\left(\xi^K\cdot \bold{Z}\right)_T
$,
where
$K_0$ is $\mathcal{A}_0$-measurable and
$\xi^K$ is an $m$-dimensional
predictable  process
 such that the analogous of
(\ref{eq-integrands}) holds (see pages 27-28 in  \cite{jacod}).\bigskip\\
To conclude the \textit{PRP is invariant by change of probability measures} in the following sense.~For a probability measure $\tilde{R}$ equivalent to $R|_{\mathcal{A}_T}$, set
$$L_t:=\frac{d\tilde{R}}{dR}\Big|_{\mathcal{A}_t},\;\;t\leq T.$$
$L=(L_t)_{t\in[0,T]}$ is a positive $(R,\mathbb{A})$-martingale and the following result holds.
\begin{proposition}\label{lemma:invariance}(Lemma 2.5 in \cite{jean-song15})\\
Let $\boldsymbol{Z}=(Z_1,...,Z_m)$, with $Z_i\in\mathcal{M}_{loc}(R,\mathbb{A})$, enjoy the $(R,\mathbb{A})$-PRP.~Assume that for any component $Z^i$ of $\boldsymbol{Z}$ there exists $\langle L,Z^i\rangle^{R,\mathbb{A}}$.~Set
$$\tilde{Z}^i_t:=Z^i_t - \int_0^t\,\frac1{L_{s^{-}}}\,d\langle L,Z^i\rangle^{R,\mathbb{A}}_s.$$
Then $\boldsymbol{\tilde{Z}}$ enjoys the $(\tilde{R},\mathbb{A})$-PRP.
\end{proposition}
\indent Now we introduce the notion of \textit{Jacod's dimension of the set $\mathcal{H}^1(R,\mathbb{A})$}, where
$$\mathcal{H}^1(R,\mathbb{A}):=\{M\in \mathcal{M}(R,\mathbb{A}): E^R[sup_{t\in [0,T]}|M_t|]<+\infty\}.$$
\noindent Let $\boldsymbol{Z}$ be a vector local martingale.~We denote by $\mathcal{Z}^1(\boldsymbol{Z},R)$  the \textit{stable subspace in $\mathcal{H}^1(R,\mathbb{A})$ generated by $\boldsymbol{Z}$} (see Definition (4.4), page 114,  in \cite{jacod}).
\begin{definition}\label{def_dimension} (Definition (4.38), pages 130-131, in \cite{jacod})\\
$dim\big(\mathcal{H}^1(R,\mathbb{A})\big)$ is the minimal dimension of a 1-generator of $\mathcal{H}^1(R,\mathbb{A})$, where a 1-generator of $\mathcal{H}^1(R,\mathbb{A})$ is any vector local martingale $\boldsymbol{Z}$ such that $\mathcal{Z}^1(\boldsymbol{Z},R)=\mathcal{H}^1_0(R,\mathbb{A})$.
\end{definition}
\noindent By Theorem 13.4, page 363, in \cite{he-wang-yan92} it follows immediately that
\textit{$dim\big(\mathcal{H}^1(R,\mathbb{A})\big)$ is the minimal possible dimension of a vector-valued local martingale enjoying the $(R,\mathbb{A})$-PRP.}
\begin{proposition}\label{thm-dimension}(Proposition 2.10 in \cite{aksamit-fontana-19})\\
If $\tilde{R}$ is a probability measure on $(\Omega, \mathcal{A})$ equivalent to $R$, then $dim\big(\mathcal{H}^1(\tilde{R},\mathbb{A})\big)=dim\big(\mathcal{H}^1(R,\mathbb{A})\big)$.
\end{proposition}
 For the definition of \textit{basis} and \textit{multiplicity} of a filtration we follow \cite{davis1}.
 \begin{definition}\label{def:basis}
An $(R,\mathbb{A})$-basis is a (at most countable) subset $\{Z_1, Z_2, \ldots\}$ of $\mathcal{M}^2(R,\mathbb{A})$ whose elements are pairwise strongly orthogonal and
 such that each $V\in \mathcal{M}^2(R,\mathbb{A})$ satisfies
\begin{align*}
V_t=V_0+\sum_{i}\int_0^t \Phi^V_i(s) d Z_i(s)
\end{align*}
where
$\Phi^V_i$ is a
predictable process such that $E^R\left[\int_0^T\,(\Phi^V_i)^2(s) d[Z_i]_s\right]<+\infty$ and $V_0$ is $\mathcal{A}_0$-measurable.
\end{definition}
\begin{remark}\label{rem: stochastic integrals}
If  $\bold{Z}=(Z_1,...,Z_m)$ with $Z_i\in\mathcal{M}^2(R,\mathbb{A})$ enjoys the $(R,\mathbb{A})$-PRP and has pairwise strongly orthogonal components then $(Z_1,...,Z_m)$ is a $(R,\mathbb{A})$-basis.
 \end{remark}
\begin{definition}\label{def:multiplicity}
The multiplicity of the filtration $\mathbb{A}$ under the measure
$R$ is the smallest integer $k\in \mathbb{N}\cup\{+\infty\}$ such
that there exists an $(R,\mathbb{A})$-basis of dimension $k$.
\end{definition}
For stating our main result we will need the notion of \textit{minimal martingale measure}.
\begin{definition}\label{def-min-mart-meas} (Definition 1  in \cite{ans_str92})\\
Let $Y$ be a special semi-martingale with local martingale part, $Z$.~A measure $Q$ in $\mathbb{P}(Y,\mathbb{A})$ is
the minimal martingale measure for $Y$ under $R$ if
$Y$ is a $(Q,\mathbb{A})$-martingale,
$Q|\mathcal{A}_0=R|\mathcal{A}_0$ and any $V\in\mathcal{M}^2_{loc}(R,\mathbb{A})$ orthogonal to $Z$ belongs to $\mathcal{M}_{loc}(Q,\mathbb{A})$.
\end{definition}
  Finally we provide a result linking the regularity of a special semi-martingale to the regularity of the elements of its canonical decomposition.
  \begin{proposition}\label{prop-regularity}
  Let $Y$ be a square-integrable special semi-martingale and let $Z$ be its local martingale part.~Then $Z\in\mathcal{M}^2(R,\mathbb{A})$.
  \end{proposition}
  \begin{proof}
  The regularity of $Y$ implies that its quadratic variation process $[Y]$ is integrable (see Theorem 2, page 246, in \cite{Prott}) and by Theorem VII-55, page 245, in \cite{del-me-b} it follows that $[Z]$ is integrable.~Then the thesis follows applying Theorem 11.4.5, page 243, in \cite{coh-ell-15}.
  \end{proof}
\subsection{The sharp bracket of the accessible part of a martingale}
It is well-known that any adapted process of locally integrable variation admits a unique
dual predictable projection also called \textit{(predictable) compensator} (see Theorem VI-80, page 139, in \cite{del-me-b}).\bigskip\\
Here we explicitly provide the compensator of a pure jump process with accessible jump times only and
we use its expression to compute  the sharp bracket of the accessible part of any square-integrable martingale.
\begin{lemma}\label{prop-explicit-compensator}
Consider an $\mathbb{A}$-adapted pure jump process $U$ of integrable variation with accessible jump times only.~Then its
$(R,\mathbb{A})$-compensator $B^U$ admits the representation
\begin{align}\label{eq-repr-sharp-var-Z}
B^U_t=& \sum_{n\in \mathbb{N}}\,\sum_{m\in
\mathbb{N}}\,E^R\left[\Delta U_{\eta_{n,m}}\;
\mathbb{I}_{\eta_n=\eta_{n,m}}
\mid\mathcal{A}_{{\eta_{n,m}}^-}\right]\mathbb{I}_{\eta_{n,m}\le
t}
\end{align}
where  $\{\eta_n\}_{n\in \mathbb{N}}$ is the sequence of the jump
times of $U$ and, for each $n\in \mathbb{N}$, $\{\eta_{n,m}\}_{m\in\mathbb{ N}}$ is an enveloping sequence of $\eta_n$.
\end{lemma}
\begin{proof}
$U$ satisfies
\begin{align}\label{eq-repr-quadr-var-Z}
 U_t=\sum_{s\leq t}\,\Delta
U_s=\sum_{n\in
\mathbb{N}}\,\Delta U_{\eta_n}\;\mathbb{I}_{\eta_n\le t}=
\sum_{n\in \mathbb{N}}\,\sum_{m\in \mathbb{N}}\,\Delta
U_{\eta_{n,m}}\; \mathbb{I}_{\eta_n=\eta_{n,m}}
\mathbb{I}_{\eta_{n,m}\le t}.
\end{align}
Moreover
for any $\mathbb{A}$-predictable stopping time $\eta$ it has to
$$
\Delta B^U_\eta=E^R\left[\Delta
U_\eta\mid\mathcal{A}_{\eta^-}\right]
$$
(see Theorem VI-76, page 136, in \cite{del-me-b}).~Observe that the
right hand side of the above equality makes sense by the
integrability assumption on $U$.~We stress that the same assumption provides the
 integrability conditions necessary for the rest of the proof.\bigskip\\
Equality  (\ref{eq-repr-sharp-var-Z}) is proved once we show that
$$U_t-\sum_{n\in
\mathbb{N}}\,\sum_{m\in \mathbb{N}}\,E^R\left[\Delta
U_{\eta_{n,m}}\; \mathbb{I}_{\eta_n=\eta_{n,m}}
\mid\mathcal{A}_{{\eta_{n,m}}^-}\right]\;\mathbb{I}_{\eta_{n,m}\le
t},\;\;\;t\in[0,T]$$
is an $(R,\mathbb{A})$-martingale.~In fact, since the second
addend in the previous expression is an $\mathbb{A}$-predictable
process (see Lemma 22.3 ii), page 411, in \cite{kall97}), the statement follows by the
uniqueness of the compensator.
\bigskip\\
Using representation (\ref{eq-repr-quadr-var-Z}), for any $s \le
t$, we get
\begin{align*}
& E^R\left[U_t - \sum_{n\in \mathbb{N}}\,\sum_{m\in
\mathbb{N}}\,E^R\left[\Delta U_{\eta_{n,m}}\;
\mathbb{I}_{\eta_n=\eta_{n,m}}
\mid\mathcal{A}_{{\eta_{n,m}}^-}\right]\;\mathbb{I}_{\eta_{n,m}\le
t}\mid \mathcal{A}_s\right]=\\
&U_s- \sum_{n\in \mathbb{N}}\,\sum_{m\in
\mathbb{N}}\,E^R\left[\Delta U_{\eta_{n,m}}\;
\mathbb{I}_{\eta_n=\eta_{n,m}}
\mid\mathcal{A}_{{\eta_{n,m}}^-}\right]\;\mathbb{I}_{\eta_{n,m}\le
s}\;+\\
& E^R\left[ \sum_{n\in \mathbb{N}}\,\sum_{m\in
\mathbb{N}}\,\left(\Delta U_{\eta_{n,m}}\;
\mathbb{I}_{\eta_n=\eta_{n,m}} -E^R\left[\Delta U_{\eta_{n,m}}\;
\mathbb{I}_{\eta_n=\eta_{n,m}}
\mid\mathcal{A}_{{\eta_{n,m}}^-}\right]\right)\mathbb{I}_{s<\eta_{n,m}\leq
t}\mid\mathcal{A}_s\right].
\end{align*}
Last term in previous expression can be written as
\begin{align*}
\sum_{n\in \mathbb{N}}\,\sum_{m\in \mathbb{N}}\, E^R\Big[\Big(\Delta
U_{\eta_{n,m}}\mathbb{I}_{\eta_n=\eta_{n,m}}\; -E^R\left[\Delta
U_{\eta_{n,m}}\mathbb{I}_{\eta_n=\eta_{n,m}}\;
\mid\mathcal{A}_{{\eta_{n,m}}^-}\right]\Big)\mathbb{I}_{\eta_{n,m}\le
t}\mid\mathcal{A}_s\Big]\,\mathbb{I}_{s<\eta_{n,m}}
\end{align*}
Since
$\sigma(\eta_{n,m})\subset\mathcal{A}_{{\eta_{n,m}}^-}$ (see
Chapter III Theorem 3.4 point 1, page 80,  in \cite{he-wang-yan92}), then the general
term of last sum is null if and only if
\begin{align*}
&E^R\left[ \Delta
U_{\eta_{n,m}}\mathbb{I}_{\eta_n=\eta_{n,m}}\,\mathbb{I}_{\eta_{n,m}\le
t}\mid\mathcal{A}_s\right]\,\mathbb{I}_{s<\eta_{n,m}}=\\
=&E^R\Big[
E^R\left[\Delta U_{\eta_{n,m}}\
\mathbb{I}_{\eta_n=\eta_{n,m}}\,\mathbb{I}_{\eta_{n,m}\le
t}\mid\mathcal{A}_{{\eta_{n,m}}^-}\right]\mid\mathcal{A}_s\Big]\,\mathbb{I}_{s<\eta_{n,m}},
\end{align*}
that is if and only if for any set $A\in\mathcal{A}_s$
\begin{align*}
&\int_{A\cap(s<\eta_{n,m})}E^R\left[\Delta
U_{\eta_{n,m}}\mathbb{I}_{\eta_n=\eta_{n,m}}\,\mathbb{I}_{\eta_{n,m}\le
t}\mid\mathcal{A}_s\right]\,dR=\\
=&\int_{A\cap(s<\eta_{n,m})}E^R\Big[ E^R\left[\Delta
U_{\eta_{n,m}}\mathbb{I}_{\eta_n=\eta_{n,m}}\,\mathbb{I}_{\eta_{n,m}\le
t}\mid\mathcal{A}_{{\eta_{n,m}}^-}\right]\mid\mathcal{A}_s\Big]\,dR.
\end{align*}
Last equality is true.~In fact, $A\cap(s<\eta_{n,m})$ belongs either to $\mathcal{A}_s$ or to $\mathcal{A}_{{\eta_{n,m}}^-}$, so that both expressions coincide with $$\int_{A\cap(s<\eta_{n,m})}\Delta U_{\eta_{n,m}}\,\mathbb{I}_{\eta_n=\eta_{n,m}}\,\mathbb{I}_{\eta_{n,m}\le t}\,dR.$$
\end{proof}
\begin{remark}
The process $B^U$, although unique, has different representations according to the  choice of the enveloping sequences of the jump times of $U$.~Nevertheless the representation of $B^U$ as multivariate point process is unique and  is obtained from (\ref{eq-repr-sharp-var-Z}),
whichever family $\big\{\{\eta_{n,m}\}_{m\in\mathbb{ N}},\;n\in\mathbb{N}\big\}$ is chosen.~More precisely,
\begin{align}\label{eq-repr-sharp-multi}
B^U_t=& \sum_{n\in \mathbb{N}}\,\sum_{m\in
\mathbb{N}}\,E^R\left[\Delta U_{\widetilde{\eta}_{n,m}}\;
\mathbb{I}_{\eta_n=\widetilde{\eta}_{n,m}}
\mid\mathcal{A}_{{\widetilde{\eta}_{n,m}\;^-}}\right]\mathbb{I}_{{\widetilde{\eta}_{n,m}}\le
t},
\end{align}
where
$$
\widetilde{\eta}_{n,m}=\begin{cases}
\eta_{n,m} &  \text{ if }  E^R\left[\Delta U_{\eta_{n,m}}\;
\mathbb{I}_{\eta_n=\eta_{n,m}}
\mid\mathcal{A}_{{\eta_{n,m}}^-}\right]\neq 0;\\
+\infty   &  \text{ if }  E^R\left[\Delta U_{\eta_{n,m}}\;
\mathbb{I}_{\eta_n=\eta_{n,m}}
\mid\mathcal{A}_{{\eta_{n,m}}^-}\right]=0,
\end{cases}
$$
and the sequence of increasing stopping times and marks of $B^U$ is identified by reordering the set $\{\widetilde{\eta}_{n,m}\}_{n,m\in\mathbb{ N}}$.~Note that, fixed $n$, the set
$\{\widetilde{\eta}_{n,m}\}_{m\in\mathbb{ N}}$ is an enveloping sequence of $\eta_{n}$.~In particular, for any $m$, $\widetilde{\eta}_{n,m}$ is a predictable random time by Proposition 2.10, page 17, in \cite{ja-sh03}.
\end{remark}
\begin{proposition}\label{cor-explicit-compensator}
If $Z\in \mathcal{M}^2(R,\mathbb{A})$, then
\begin{align*}
\langle Z^{dp}\rangle^{R,\mathbb{A}}_t=& \sum_{n\in
\mathbb{N}}\,\sum_{m\in \mathbb{N}}\,E^R\left[(\Delta
Z_{\eta_{n,m}})^2\; \mathbb{I}_{\eta^{dp}_n=\eta_{n,m}}
\mid\mathcal{A}_{{\eta_{n,m}}^-}\right]\mathbb{I}_{\eta_{n,m}\le
t},
\end{align*}
where $Z^{dp}$ is the accessible martingale part of $Z$ in the
decomposition  (\ref{eq-decomp-yoeurp}),
$\{\eta^{dp}_n\}_{n\in \mathbb{N}}$ is the sequence of the jump
times of $Z^{dp}$ and, for each $n\in \mathbb{N}$,
$\{\eta_{n,m}\}_{m\in\mathbb{ N}}$ is an enveloping sequence of $\eta^{dp}_n$.
\end{proposition}
\begin{proof}
We recall that $\langle Z^{dp}\rangle^{R,\mathbb{A}}$ exists since $Z^{dp}\in \mathcal{M}^2(R,\mathbb{A})$ and it is the compensator of $[Z^{dp}]$.~Therefore it is enough to apply previous lemma to
$[Z^{dp}]$, which is the adapted purely discontinuous integrable increasing process with accessible jump times only which satisfies
\begin{equation*}
   [Z^{dp}]_t=\sum_{s\leq t}\,(\Delta
Z^{dp}_s)^2=\sum_{n\in
\mathbb{N}}\,\sum_{m\in \mathbb{N}}\,(\Delta
Z^{dp}_{\eta_{n,m}})^2\mathbb{I}_{\eta_{n,m}\le
t}=\sum_{n\in
\mathbb{N}}\,\sum_{m\in \mathbb{N}}\,(\Delta
Z_{\eta_{n,m}})^2\mathbb{I}_{\eta_{n,m}\le
t}.
\end{equation*}
\end{proof}
\subsection{The compensated occurrence process of a random time}\label{subsec-compensated op}
Let $\eta$ be a random time in $[0,+\infty]$ and consider the occurrence process of $\eta$ restricted to $[0,T]$, that is
the process  $$\mathbb{I}_{\eta\le
\cdot}:=(\mathbb{I}_{\eta\le t})_{t\in [0,T]}.$$
It is a pure jump bounded process with a unique jump time corresponding to
$$
\eta \;\mathbb{I}_{\eta\le T} +\infty\;\mathbb{I}_{\eta>T}.
$$
\textit{In the rest of the paper, for notational convenience, we will identify it with $\eta$}.\bigskip\\
Define
$$\mathcal{H}^{\eta}_t:=\sigma(\eta\wedge t).$$ Then $\mathbb{H}^{\eta}=(\mathcal{H}^{\eta}_t)_{t\in [0,T]}$ is the natural filtration
of $\mathbb{I}_{\eta\le \cdot}$, that is the smallest
filtration which makes $\eta$ a stopping time.~Let $A^{\eta,R,\mathbb{H}^{\eta}}$ be the
$(R,\mathbb{H}^{\eta})$-compensator of $\mathbb{I}_{\eta\le
\cdot}$ or simply the \textit{(natural) compensator of $\eta$} and let
$H^{\eta}=(H^{\eta}_t)_{t\in [0,T]}$ be the \textit{(naturally) compensated occurrence process of $\eta$}, that is
the $(R,\mathbb{H}^{\eta})$-martingale defined by
\begin{equation}\label{eq-P-default-martingale}
 H^{\eta}_t:=\mathbb{ I}_{\eta\le t}- A^{\eta,R,\mathbb{H}^{\eta}}_t.
\end{equation}
\begin{remark}\label{caratterizzazione dei tempi aleatori}
It is to stress that any random time $\eta$ in $[0,T]\cup \{+\infty\}$, according to its law $\mu^\eta$, is an
$\mathbb{H}^{\eta}$-totally inaccessible stopping time if and only if
$\mu^\eta$ restricted to $[0,T]$ is a diffusive measure and it is an $\mathbb{H}^{\eta}$-accessible
stopping time if and only if $\mu^\eta$ is atomic.~This is a trivial generalization of Theorem IV-107, page 241, in
\cite{del-me-a}, which only covers the case when $\eta$ is finite.~Therefore $\eta$, unless trivial, can never be $\mathbb{H}^{\eta}$-predictable,
so that the martingale $H^{\eta}$ cannot be identically null.
\end{remark}
According to  (\ref{eq-eta-dec}) we denote by $\eta^{dp}$ and $\eta^{dq}$  the $\mathbb{H}^{\eta}$-accessible and the $\mathbb{H}^{\eta}$-totally
inaccessible
component of $\eta$, respectively.~Next proposition highlights their different contribution in the expression of the compensator of $\eta$.
\begin{proposition}
The compensated
occurrence process of $\eta$ admits the representation
\begin{equation}\label{H-prime}
H^{\eta}_t=\mathbb{I}_{\{\eta\le t\}}-A^{dq,R,\mathbb{H}^{\eta}}_t-\sum_{m\in \mathbb{N}}R(\eta^{dp}=\eta^{dp}_m\mid\mathcal{H}^{\eta}_{{{\eta^{dp}_m}^-}})\mathbb{ I}_{\{\eta^{dp}_m\le t\}},
\end{equation}
where $A^{dq,R,\mathbb{H}^{\eta}}$ denotes the (continuous) $(R,\mathbb{H}^{\eta})$-compensator
of $\eta^{dq}$ and $(\eta^{dp}_m)_{m\in \mathbb{N}}$ is an enveloping sequence of $\eta^{dp}$.
\end{proposition}
\begin{proof}
We start by the simple equality
\begin{equation}\label{decomposition of occurrence process}
\mathbb{I}_{\eta\le t}=\mathbb{I}_{\eta=\eta^{dq}}\mathbb{I}_{\eta^{dq}\le t}+\mathbb{I}_{\eta=\eta^{dp}}\mathbb{I}_{\eta^{dp}\le t}.
\end{equation}
Observing that
$$\{\eta^{dq}\le t\}\subset \{\eta=\eta^{dq}\}, \ \ \{\eta^{dp}\le t\}\subset \{\eta=\eta^{dp}\},$$
we get
\begin{equation}\label{decomposizione H}
\mathbb{I}_{\eta\le t}=\mathbb{I}_{\eta^{dq}\le t}+\mathbb{I}_{\eta^{dp}\le t}.
\end{equation}
In analogy with the notation $A^{dq,R,\mathbb{H}^{\eta}}$, we indicate with $A^{dp,R,\mathbb{H}^{\eta}}$ the $(R,\mathbb{H}^{\eta})$-compensator
 of  $\eta^{dp}$.~Then by the uniqueness of the compensator of $\eta$ we derive
\begin{equation}\label{eq-dec-comp-H}
 A^{\eta,R,\mathbb{H}^{\eta}}=A^{dq,R,\mathbb{H}^{\eta}}+A^{dp,R,\mathbb{H}^{\eta}}.
\end{equation}
Therefore
\begin{equation}\label{decomposizione G H}
H^{\eta}=\left(\mathbb{I}_{\eta^{dq}\le \cdot}-A^{dq,R,\mathbb{H}^{\eta}}\right)+\left(\mathbb{I}_{\eta^{dp}\le \cdot}-A^{dp,R,\mathbb{H}^{\eta}}\right),
\end{equation}
that is
\begin{equation*}
H^{\eta}=\mathbb{I}_{\eta\le \cdot}-A^{dq,R,\mathbb{H}^{\eta}}-A^{dp,R,\mathbb{H}^{\eta}}.
\end{equation*}
Finally Lemma \ref{prop-explicit-compensator} yields
\begin{equation}\label{eq-pred-comp-tau}
A^{dp,R,\mathbb{H}^{\eta}}_t=\sum_{m\in \mathbb{N}}R(\eta^{dp}=\eta^{dp}_m\mid\mathcal{H}_{{{\eta^{dp}_m}^-}})\mathbb{ I}_{\{\eta^{dp}_m\le t\}}.
\end{equation}
\end{proof}
We stress that, since the two martingales at the right-hand side of (\ref{decomposizione G H}) are strongly orthogonal, by the uniqueness of Yoeurp's decomposition (\ref{eq-decomp-yoeurp}), the $\mathbb{H}$-totally inaccessible and the $\mathbb{H}$-accessible  martingale  parts of $H$ satisfy the equalities
$$H^{dq}_t=\mathbb{I}_{\eta^{dq}\le t}-A^{dq,R,\mathbb{H}^{\eta}}_t, \ \ \ \ H^{dp}_t=\mathbb{I}_{\eta^{dp}\le t}-A^{dp,R,\mathbb{H}^{\eta}}_t.$$
We conclude recalling a key result to be used in the rest of the paper.
\begin{theorem}\label{prop-H-prp} (Proposition 2 in \cite{chou-meyer-75})\\
$H$ enjoys the
 $(R,\mathbb{H}^{\eta})$-PRP.
\end{theorem}
\subsection{Mutual singularity of the kernels associated with two predictable increasing processes}
In this subsection we prove a result on the supports of the random measures generated by two increasing processes, which will be the core of the study of the Jacod's dimension of  the space of martingales on the enlarged filtration $\mathbb{G}$.\bigskip\\
Let us start with two definitions given in analogy
with those at pages 19 and 374 in \cite{Billy}, respectively.
\begin{definition}
A \textit{support of a measure} $\mu$ on a measurable space $(E,
\mathcal{E})$ is any set $C\in \mathcal{E}$ such that
$\mu(E\setminus C)=0$.
\end{definition}
\begin{definition}
Two measures $\mu$ and $\nu$ on the measurable space $(E,
\mathcal{E})$ are \textit{mutually singular} if they admit disjoint
supports, that is if there exist two measurable disjoint sets
$C^\mu$ and $C^\nu$ such that $\mu(E\setminus C^\mu)=0$ and
$\nu(E\setminus C^\nu)=0$.
\end{definition}
\begin{remark}
If $\mu$ and $\nu$ are mutually singular with disjoint supports  $C^\mu$ and $C^\nu$, then $\mu(C^\nu)=\nu(C^\mu)=0$.
\end{remark}
We call
\textit{random set} any subset of the product space $\Omega\times
[0,T]$,
 \textit{measurable random set}  any random set belonging
to the product sigma-algebra $\mathcal{ A}\otimes
\mathcal{B}([0,T])$ and \textit{predictable random set} any random
set belonging to the predictable sigma algebra $\mathcal{P}(\mathbb{A})$ on
$\Omega\times [0,T]$
(see pages 3 and 16 in \cite{ja-sh03}).~Finally given a random
set $C$ we denote by $\overline{C}$ its complementary set and
by $C^\omega$ its $\omega$-section, that is
\begin{equation}\label{eq-omega-section}
C^\omega:=\{t,\, (t,\omega)\in C\}.
\end{equation}
Moreover we also introduce
next definition (see pages 29 and 30 in \cite{kall-17}).
\begin{definition}\label{def-kernel}
A function $\kappa:\Omega\times \mathcal{B}([0,T])\rightarrow
\mathbb{R}^+$,
$(\omega, G)\rightarrow \kappa(\omega, G)$, is a kernel from
$\Omega$ to $[0,T]$ if and only if,  for any fixed $G\in
\mathcal{B}([0,T])$, $\omega\rightarrow\kappa(\omega, G)$ is
$\mathcal{A}$-measurable and, for any fixed $\omega$,
$G\rightarrow \kappa(\omega, G)$ is a measure on $\left([0,T],\mathcal{B}([0,T])\right)$.\\
Two kernels $\kappa^i$, $i=1,2$,  from $\Omega$ to $[0,T]$ are
mutually singular if there exist two disjoint  measurable random
sets $\Gamma_i$, $i=1,2$, with $\Gamma_i^\omega$ support of the measure $\kappa^i(\omega, \cdot)$, that is $\kappa^i(\omega,
\overline{\Gamma^\omega_i})=0$, $i=1,2$.
\end{definition}
\begin{remark}\label{rem-kernel-induced}
Observe that any increasing process $(V_t)_{t\in[0,T]}$
uniquely defines a kernel
$dV:\Omega\times \mathcal{B}([0,T])\rightarrow \mathbb{R}^+$,
$(\omega, G)\rightarrow \int_GdV_s(\omega)$  (for further details
see VI-86, page 145,  in \cite{del-me-b}).
\end{remark}
\begin{proposition}\label{lemma-disjoint-pred-supp}
Let $(A_t)_{t\in[0,T]}$ and  $(B_t)_{t\in[0,T]}$ be two
non-negative increasing predictable processes.~Assume that the associated
kernels $dA$ and $dB$ are  mutually singular.~Then there exists a
predictable random set $C^A$ such that, $R$-a.s., $C^{A,\omega}$
and $\overline{C^{A,\omega}}$ are disjoint supports
 of the measures $dA(\omega, \cdot)$ and $dB(\omega, \cdot)$, respectively.
\end{proposition}
\begin{proof}
By assumption there exist two disjoint measurable random sets
$\Gamma^A$ and $\Gamma^B$ such that, for any $\omega$,
$dA(\omega,\overline{\Gamma^{A,\omega}})=dB(\omega,\overline{\Gamma^{B,\omega}})=0$ and in particular
 \begin{equation}\label{eq-disjoint-support}
   dB(\omega,\Gamma^{A,\omega})=0.
\end{equation}
First, by means of a projection procedure, we prove that it is
possible to construct a predictable random set $C^A$ whose
$\omega$-sections are $R$-a.s.~supports of the measures $dA(\omega, \cdot)$.\\
Then, we prove that $R$-a.s.~the $\omega$-section   of the predictable
random set $\overline{C^{A}}$ is support of the
measure $dB(\omega, \cdot)$ and, since  $(\overline{C^{A}})^\omega=\overline{C^{A,\omega}}$, this ends the proof.~Let us be
more precise.\bigskip\\
We construct $C^A$ as  the \textit{predictable  support} of
the measurable random set  $\Gamma^A$, that is
\begin{equation}\label{eq-predictable-supp}
C^A:=\big\{(\omega,t), \ \
^{p\,}\mathbb{I}_{\Gamma^A}(\omega,t)>0\big\},
\end{equation}
where $^{p\,}\mathbb{I}_{\Gamma^A}$ denotes the predictable
projection of the process $\mathbb{I}_{\Gamma^A}$ (see Chapter 1
Definition 2.32, page 24, in \cite{ja-sh03}).~Now we prove
 that $R$-a.s.~the $\omega$-section  $C^{A,\omega}$ is a support of the
measure $dA_\cdot(\omega)$.\bigskip\\
 To this end it suffices to establish that
\begin{equation}\label{eq-newsupport}
\int_{\overline{{C}^{A,\omega}}}dA_s(\omega)=0,  \ \ R\text{-a.s.}
\end{equation}
 In fact $\int_{
\overline{{C}^{A,\omega}}}dA_s(\omega)\ge 0$ and moreover the
following equalities hold
\begin{equation}\label{eq-complementari}
E^R\left[\int_{
\overline{{C}^{A,\cdot}}}dA_s(\cdot)\right]=E^R\left[\int_{
\overline{{C}^{A,\cdot}}}\mathbb{I}_{{\Gamma^{A}}}(\cdot,s)dA_s(\cdot)\right]=E^R\left[\int_{[0,T]}\mathbb{I}_{
\overline{{C}^{A,\cdot}}}(s)\mathbb{I}_{{\Gamma^{A}}}(\cdot,s)dA_s(\cdot)\right].
\end{equation}
Observe that for any $(\omega, t)\in\Omega\times [0,T]$ it holds
\begin{equation*}
\mathbb{I}_{ \overline{{C}^{A,\omega}}}(t)=\mathbb{I}_{
\overline{{C}^{A}}}(\omega, t).
\end{equation*}
 Then the last term in (\ref{eq-complementari}) coincides with
\begin{equation}\label{eq-zeta}
    E^R\left[\int_{[0,T]}\zeta(\cdot,s)dA_s(\cdot)\right],
\end{equation}
where $\zeta(\omega,t):=\mathbb{I}_{ \overline{{C}^{A}}}(\omega,
t)\mathbb{I}_{{\Gamma^{A}}}(\omega,t)$.~Again,  due to the fact
that, by assumption, the process $A$ is predictable, the
expression in (\ref{eq-zeta})  turns to be equal to (see Theorem
VI.57, page 122,  in \cite{del-me-b})
\begin{align*}
&E^R\left[\int_{[0,T]}\, ^p\zeta(\cdot, s)dA_s(\cdot)\right],
\end{align*}
where, being $\overline{C^A}$ a predictable random set,
$$^p\zeta(\omega, t)=\mathbb{I}_{ \overline{{C}^{A}}}(\omega,
t)\,^{p\,}\mathbb{I}_{{\Gamma^{A}}}(\omega,t)$$
 (see Chapter 1 Theorem 2.28 c) in \cite{ja-sh03}).~Then
\begin{align*}
&E^R\left[\int_{[0,T]}\, ^p\zeta(\cdot,s)dA_s(\cdot)\right]=
E^R\left[\int_{ \left(\overline{{C}^{A}}\right)^\cdot}\,
^{p\,}\mathbb{I}_{{\Gamma^A}}(\cdot,s)dA_s(\cdot)\right].
\end{align*}
Since by
definition for any $\omega$ the function  the function $
^{p\,}\mathbb{I}_{{\Gamma^A}}(\omega,\cdot)$ is null on the set
$\overline{{C}^{A,\omega}}$, the right hand side in previous equality is equal to zero, so that
(\ref{eq-newsupport}) follows.\bigskip\\
Let us prove that $\overline{{C}^{A,\omega}}$ is a
support of $dB(\omega, \cdot)$.~This follows as soon as we show that  $R$-a.s.
 \begin{equation}\label{eq-disjoint}
\int_{{C}^{A,\omega}}dB_s(\omega)=0.
\end{equation}
 In fact
\begin{equation*}
   0= E^R\left[\int_{\Gamma^A}dB_s\right]=E^R\left[\int_{[0,T]}
    \mathbb{I}_{\Gamma^{A,\cdot}}(s)dB_s(\cdot)\right]=E^R\left[\int_{[0,T]}\,^{p\,}\mathbb{I}_{\Gamma^A}(\cdot,s)dB_s(\cdot)\right],
\end{equation*}
where last equality again follows by Theorem VI.57 in \cite{del-me-b}.~Therefore
$$
E^R\left[\int_{C^{A,\cdot}}\,^{p\,}\mathbb{I}_{\Gamma^A}(\cdot,s)dB_s(\cdot)\right]=0
$$
which is equivalent to
$\int_{C^{A,\omega}}\,^{p\,}\mathbb{I}_{\Gamma^A}(\omega,s)dB_s(\omega)=0$,
$R$-a.s.~Since by definition, fixed $\omega$, $^{p\,}\mathbb{I}_{\Gamma^A}(\omega,s)>0$ for all
$s\in C^{A,\omega}$,
then (\ref{eq-disjoint}) follows, $R$-a.s.
\end{proof}
 \section{The setting}
In the following we introduce the setting we will work in looking for a new martingale representation theorem.\bigskip\\
From now on we will consider a fixed probability space $(\Omega,\mathcal{F}, P)$ and  two standard
filtrations on it,
$\mathbb{F}=(\mathcal{F}_t)_{t\in [0,T]}$ and
$\mathbb{H}=(\mathcal{H}_t)_{t\in [0,T]}$, with trivial initial $\sigma$-algebras and such that $\mathcal{F}_T\subset\mathcal{F}$ and
$\mathcal{H}_T\subset\mathcal{F}$.~We will set
\begin{equation}\label{def-G}\mathbb{G}:=\mathbb{F}\vee\mathbb{H}.\end{equation}
\textit{From now on we will standing assume the following condition $\textrm{(D)}$.}\vspace{0.5em}\\
$\textrm{(D)}$  There exists an \textit{equivalent decoupling measure} that is a probability measure $Q$ on $(\Omega, \mathcal{G}_T)$ equivalent to $P|_{\mathcal{G}_T}$ under which $\mathcal{F}_T$ and $\mathcal{H}_T$ are independent.\bigskip\\
Then also $\mathbb{G}$ under $P$ is a standard filtration (see Lemma 2.2 in \cite
{ame-be-schw03}).\bigskip\\
Let $M=(M_t)_{t\in [0,T]}$ and
$N=(N_t)_{t\in [0,T]}$ be a square-integrable
 $(P,\mathbb{F})$-martingale and a square-integrable
 $(P,\mathbb{H})$-martingale respectively.
We define on $(\Omega,\mathcal{G}_T)$ the \textit{martingale preserving measure} measure $P^*$ associated to $Q$ by
\begin{equation}\label{eq-mgpreserving-decoupl-measure}
dP^*:=\frac{dP}{dQ}\Big|_{\mathcal{F}_T}\cdot\frac{dP}
 {dQ}\Big|_{\mathcal{H}_T}\,dQ.
 \end{equation}
Then $M$ and $N$ are independent $(P^*,\mathbb{G})$-martingales and each of them
 preserves  under $P^*$ its law since
 \begin{equation}\label{eq:P-marginals}P^*|_{\mathcal{F}_T}=P|_{\mathcal{F}_T}\;\;\;\; P^*|_{\mathcal{H}_T}=P|_{\mathcal{H}_T}.\end{equation}
In the light of Theorem \ref{thm-decomp-yoeurp} we consider the
$(P^*,\mathbb{G})$-decompositions of $M$ and  $N$
\begin{equation}\label{Yoeurp-decomposition}
M=M^{c}+M^{dp}+M^{dq},\;\;\;N=N^{c}+N^{ dp}+N^{ dq}.
\end{equation}
\begin{proposition}\label{rem-inaccessible-time-from-F-to_G}
All $\mathbb{F}$-totally
inaccessible stopping times and all $\mathbb{H}$-totally
inaccessible stopping times are also $\mathbb{G}$-totally
inaccessible.
\end{proposition}
\begin{proof}
It is enough to prove the statement for $\mathbb{F}$ and $M$.\\
Let $\eta$ be an $\mathbb{F}$-totally inaccessible stopping
time.~Then, since condition $\textrm{(D)}$ is in force, Remark \ref{rem:total-inacess} ensures $\eta$  to avoid $\mathbb{H}$-stopping times and therefore $\eta$ turns out to be $\mathbb{G}$-totally inaccessible  (see Lemma 3.5 in \cite{jean-coku-nike12}).
\end{proof}
Recalling that
all $\mathbb{F}$-accessible random times are also
$\mathbb{G}$-accessible, by previous proposition we immediately derive next result.
\begin{corollary}\label{cor--inaccessible-time-from-F-to_G}
The $(P,\mathbb{F})$-decomposition
of $M$ and the $(P,\mathbb{H})$-decomposition of $N$ coincide with
the decompositions given in (\ref{Yoeurp-decomposition}).
\end{corollary}
\begin{proposition}\label{overlap}
For all $t$ in $[0,T]$ it
holds $P^*$-a.s.~and $P$-a.s.
\begin{equation}\label{eq-covariation0}[M,N]_t=[M^{ dp},N^{dp}]_t.\end{equation}
\end{proposition}
\begin{proof}
By the representation formula (\ref{Yoeurp-decomposition}), the linearity of the covariation operator and its invariance
under equivalent changes of measures it follows $P^*$-a.s.~and $P$-a.s.
\begin{equation*}
[M,N]=[M^{ c},N]+[M^{dp},N]+[M^{ dq},N].
\end{equation*}
In particular $M^{c}$ and $N$ as well as $M^{dq}$ and $N$ are
$(P^*,\mathbb{G})$- independent martingales.~Then, as well known,
one has $[M^{c},N]=0$ (see Theorem \ref{thm-decomp-yoeurp} and
Corollary 4.55 points b) and c) in \cite{ja-sh03}).~By Remark \ref{rem:total-inacess}
 it also holds
$[M^{dq},N]=0$.~Summarizing
\begin{equation*}
[M,N]=[M^{dp},N]
\end{equation*}
and similarly, by using the representation
(\ref{Yoeurp-decomposition}) for $N$, we get the thesis.
\end{proof}
It is worthwhile to note that (\ref{eq-covariation0}) can be explicitly written as
\begin{equation}\label{ref-rep}
[M,N]_t=[M^{dp},N^{dp}]_t=\sum_{n,l}\Delta
M^{dp}_{\eta^{dp}_{n}}\Delta
N^{dp}_{\tau^{dp}_{l}}\mathbb{I}_{{\eta^{dp}_{n}}={\tau^{dp}_{l}}}\;\mathbb{I}_{{\eta^{dp}_{n}}\le
t},
\end{equation}
 where $(\eta^{dp}_n)_{n}$ and $(\tau^{dp}_l)_{l}$ are
 sequences of  $\mathbb{G}$-accessible
 jump times of $M$ and $N$, respectively.\bigskip\\
In the previous proposition  the covariation process $[M,N]$ is
derived by using the $(P^*,\mathbb{G})$-decompositions defined in
(\ref{Yoeurp-decomposition}), which coincide with the $(P,\mathbb{F})$-decomposition of $M$ and the $(P,\mathbb{H})$-decomposition of $N$ (see Corollary
\ref{cor--inaccessible-time-from-F-to_G}).~It is to note that $\mathbb{F}$ ($\mathbb{H}$) could be different from the natural filtration $\mathbb{F}^M$ of $M$ ($\mathbb{F}^N$ of $N$).~When using the natural filtrations different Yoeurp's decompositions for $M$ and $N$ could arise.~As an example, a jump time of $M$
could be
 a $\mathbb{G}$-accessible stopping time and an $\mathbb{F}^M$-totally inaccessible one,
 and this would imply that $M^{dp}\neq M^{\mathbb{F}^M,
 dp}$.~However, as we will show now, the covariation process $[M,N]$ still coincides with the covariation process of the natural accessible martingale parts.
\begin{lemma}\label{accessible times}
Let $\eta$ and $\tau$ be an
$\mathbb{F}^M$-stopping time and an $\mathbb{F}^N$-stopping time,
respectively.~If
$P(\eta=\tau)>0$ then there exist an $\mathbb{F}^M$-accessible
stopping time $\eta^{dp}$ and an $\mathbb{F}^N$-accessible
 stopping time $\tau^{dp}$ such that on the set
$(\eta=\tau)$ it holds $\eta \,=\eta^{dp}$ and $\tau
\,=\tau^{dp}$, $P$-a.s.
\end{lemma}
\begin{proof}
The representation of $\eta$ given by (\ref{eq-eta-dec}) and (\ref{eq-time-decomposition}) yields
\begin{align*}
\eta\,\mathbb{I}_{(\eta=\tau)}=\eta\,\mathbb{I}_{(\eta=\tau=+\infty)}+\eta\,\mathbb{I}_{(\eta=\tau)\cap
A}+\eta\,\mathbb{I}_{(\eta=\tau)\cap B}.
\end{align*}
Since $A=(\eta=\eta^{dp}< +\infty)$ and $B=(\eta=\eta^{dq}< +\infty)$ and therefore $(\eta=\eta^{dp}=\eta^{dq}=+\infty)$ on $(A\cup B)^c$,
one immediately derives that $P$-a.s.
\begin{align*}
\eta\,\mathbb{I}_{(\eta=\tau)}=\eta^{dp}\,\mathbb{I}_{(\eta=\tau=+\infty)}+\eta^{dp}\,\mathbb{I}_{(\eta^{dp}=\tau)\cap
A}+\eta^{dq}\,\mathbb{I}_{(\eta^{dq}=\tau)\cap B}.
\end{align*}
Consider now the equivalent decoupling measure $Q$ introduced in condition \textrm{(D)}.~By Remark \ref{rem:total-inacess} $(\eta^{dq}=\tau)\cap B$ has null measure under  $Q$ so that a.s.
\begin{align*}
\eta\,\mathbb{I}_{(\eta=\tau)}=\eta^{dp}\,\mathbb{I}_{(\eta=\tau=+\infty)}+\eta^{dp}\,\mathbb{I}_{(\eta^{dp}=\tau)\cap
A}
\end{align*}
which implies $\eta \,=\eta^{dp}$  on $(\eta=\tau)$.\bigskip\\
Analogously one shows that $\tau \,=\tau^{dp}$  on
$(\eta=\tau)$.
\end{proof}
\begin{proposition}
For all $t$ in $[0,T]$ it
holds $P^*$-a.s.~and $P$-a.s.
\begin{equation*}
[M,N]_t=[M^{\mathbb{F}^M,dp},N^{\mathbb{F}^N,dp}]_t
\end{equation*}
where $M^{\mathbb{F}^M,dp}$ and $N^{\mathbb{F}^N,dp}$ are the accessible martingale parts in the $(P,\mathbb{F}^M)$ and the
$(P,\mathbb{F}^N)$-decompositions of $M$ and $N$, respectively.
\end{proposition}
\begin{proof}
If $\eta$ and $\tau$ are jump times  of $M^{dp}$ and $N^{dp}$, respectively, then they are
$\mathbb{F}^M$-stopping time and $\mathbb{F}^N$-stopping time,
respectively.~If $P(\eta=\tau)>0$ then, by Lemma
\ref{accessible times}, there exist an $\mathbb{F}^M$-accessible
stopping time of $M$, $\hat{\eta}^{dp}$, and an
$\mathbb{F}^N$-accessible stopping time of $N$, $\hat{\tau}^{dp}$,
such that   $P$-a.s.
$$\eta \mathds{1}_{\eta=\tau}=\hat{\eta}^{dp} \mathds{1}_{\eta=\tau} \text{  and  }  \tau \mathds{1}_{\eta=\tau}=\hat{\tau}^{dp} \mathds{1}_{\eta=\tau}.$$
Moreover
$\hat{\eta}^{dp}$ is a jump times of the accessible martingale part $M^{\mathbb{F}^M,
dp}$ in the $(P,\mathbb{F}^M)$-decomposition of $M$ and  $\hat{\tau}^{dp}$ is a jump times of the accessible martingale part $N^{\mathbb{F}^N,
dp}$ in the $(P,\mathbb{F}^N)$-decomposition of $N$.~Then using (\ref{ref-rep})
\begin{align*}
[M,N]_t=&\sum_{n,l}\mathbb{I}_{{\hat{\eta}^{dp}_n}\le
t}\mathbb{I}_{{\hat{\eta}^{dp}_n}={\hat{\tau}^{dp}_l}}\;\Delta
M^{dp}_{\hat{\eta}^{dp}_n}\Delta
N^{dp}_{\hat{\tau}^{dp}_l}=\notag\\
=&\sum_{n,l}\mathbb{I}_{{\hat{\eta}^{dp}_n}\le
t}\mathbb{I}_{{\hat{\eta}^{dp}_n}={\hat{\tau}^{dp}_l}}\;\Delta
M^{\mathbb{F}^M, dp}_{\hat{\eta}^{dp}_n}\Delta
N^{\mathbb{F}^N,dp}_{\hat{\tau}^{dp}_l}=[M^{\mathbb{F}^M,
dp},N^{\mathbb{F}^N,dp}]_t,
\end{align*}
 where $(\hat{\eta}^{dp}_n)_{n}$ is the set of all $\mathbb{F}^M$-accessible
 jump times of $M$ and $(\hat{\tau}^{dp}_l)_{l}$ is the set of all $\mathbb{F}^N$-accessible
 jump times of $N$.
\end{proof}
\begin{remark}
Under condition $\textrm{(D)}$ the covariation process $[M,N]$ is identically zero
 whenever $M$
 or $N$ are c\`adl\`ag quasi-left continuous martingales
 (see pages 121-122
 in
\cite{he-wang-yan92}).~Among the others this is the case when at least one of
the reference filtrations of $M$
 and $N$ is quasi-left
 continuous.~In fact any martingale with quasi-left continuous
reference filtration is quasi-left continuous.~In particular this happens if $M$
 or $N$ belongs to the class of
L\'evy processes without deterministic jump times (see page 190 and Exercise 8, page 148, in
\cite{Prott}).
\end{remark}
\section{The Jacod's dimension of $\mathcal{H}^1(P,\mathbb{G})$}\label{sec-predict-times}

In this section we clarify how the Jacod's dimension of $\mathcal{H}^1(P,\mathbb{G})$ is affected by the mutual behavior either of the sharp brackets of $M$ and $N$ or of their accessible parts, $M^{dp}$ and $N^{dp}$.
\begin{lemma}\label{lemma-cns-cov0}
The kernels
  $d\langle M^{ dp}\rangle^{P,\mathbb{F}}$ and $d\langle N^{dp}\rangle^{P,\mathbb{H}}$ are mutually singular if and only if the covariation process
$[M,N]$ is $P$-a.s.~null.
\end{lemma}
\begin{proof}
Let $\{\eta^{dp}_n\}_{n\in\mathbb{ N}}$ and $\{\tau^{dp}_l\}_{l\in\mathbb{ N}}$ be the sequences of  accessible
 jump times of $M$ and of $N$, respectively, and, fixed $n\in
\mathbb{N}$ and $l\in
\mathbb{N}$, let $\{\eta_{n,m}\}_{m\in\mathbb{ N}}$ and $\{\tau_{l,h}\}_{h\in\mathbb{ N}}$ be
 enveloping sequences of $\eta^{dp}_n$ and $\tau^{dp}_l$, respectively.\bigskip\\
Proposition \ref{cor-explicit-compensator} applies to
give the following representations
\begin{align}\label{eq-repr-sharp-var-M^{dp}}
\langle M^{dp}\rangle^{P,\mathbb{F}}_t=\sum_{n\in
\mathbb{N}}\,\sum_{m\in \mathbb{N}}\,U_{n,m}\,\mathbb{I}_{\eta_{n,m}\le
t},\;\;\;\;\langle N^{dp}\rangle^{P,\mathbb{H}}_t=\sum_{l\in
\mathbb{N}}\,\sum_{h\in \mathbb{N}}\,V_{l,h}\,\mathbb{I}_{\tau_{l,h}\le
t},
\end{align}
with $U_{n,m}$ and $V_{l,h}$ defined by
$$U_{n,m}:=E^P\left[(\Delta
M^{dp}_{\eta_{n,m}})^2\; \mathbb{I}_{\eta^{dp}_n=\eta_{n,m}}
\mid\mathcal{F}_{{\eta_{n,m}}^-}\right],\;\;\;V_{l,h}:=E^P\left[(\Delta
N^{dp}_{\tau_{l,h}})^2\; \mathbb{I}_{\tau^{dp}_l=\tau_{l,h}}
\mid\mathcal{H}_{{\tau_{l,h}}^-}\right].$$
Assume first the mutual singularity of the kernels
  $d\langle M^{ dp}\rangle^{P,\mathbb{F}}$ and $d\langle N^{dp}\rangle^{P,\mathbb{H}}$.~Fixed $n,m$ and $l,h$, when
$$P(\eta_{n,m}=\tau_{l,h}< +\infty)>0,$$ representations
(\ref{eq-repr-sharp-var-M^{dp}}) implies that
\begin{equation}\label{kernel singularity}\mathbb{I}_{\{\eta_{n,m}=\tau_{l,h}< +\infty\}}U_{n,m}V_{l,h}=0, \ \ P\text{-a.s. }\end{equation}
In other words, if we introduce the measurable sets
$$A_{n,m}:=\{U_{n,m}=0\},\;\;\;B_{l,h}:=\{V_{l,h}=0\},$$
then by the mutual singularity the set
$$A^c_{n,m}\cap B^c_{l,h}\cap \{\eta_{n,m}=\tau_{l,h}< +\infty\}$$
has null measure.\\
Note that equality (\ref{ref-rep}) can be written in terms of enveloping sequences as
\begin{equation}\label{ref-rep-expl}
[M,N]_t=\sum_{n,l}\sum_{m,h}\mathbb{I}_{{\eta^{dp}_{n}}=\eta_{n,m}}\;\mathbb{I}_{{\tau^{dp}_{l}}=\tau_{l,h}}\;\Delta
M^{dp}_{\eta_{n,m}}\Delta
N^{dp}_{\tau_{l,h}}\mathbb{I}_{{\eta_{n,m}}={\tau_{l,h}}}\;\mathbb{I}_{{\eta_{n,m}}\le
t}.
\end{equation}
Therefore
 $[M,N]$ is $P$-a.s.~null if for any choice of  $n,m$ and $l,h$ such that
$P(\eta_{n,m}=\tau_{l,h}< +\infty)>0$ it holds
$$\mathbb{I}_{\{\eta_{n,m}=\tau_{l,h}< +\infty\}}\mathbb{I}_{{\eta^{dp}_{n}}=\eta_{n,m}}\;\Delta
M^{dp}_{\eta_{n,m}}\mathbb{I}_{{\tau^{dp}_{l}}=\tau_{l,h}}\;\Delta
N^{dp}_{\tau_{l,h}}=0, \ \ P\text{-a.s.}.$$
Equivalently $[M,N]$ is $P$-a.s.~null if
$$\tilde{A}^c_{n,m}\cap \tilde{B}^c_{l,h}\cap \{\eta_{n,m}=\tau_{l,h}< +\infty\}$$
is a set of null measure, where
$$\tilde{A}_{n,m}:=\{\Delta
M^{dp}_{\eta_{n,m}}\,\mathbb{I}_{{\eta^{dp}_{n}}=\eta_{n,m}}=0\},\;\;\;\tilde{B}_{l,h}:=\{\Delta
N^{dp}_{\tau_{l,h}}\,\mathbb{I}_{\tau^{dp}_l=\tau_{l,h}}=0\}.$$
This easily follows  considering that up to a null measure set  it holds
$$\tilde{A}^c_{n,m}\cap \tilde{B}^c_{l,h} \subset A^c_{n,m}\cap B^c_{l,h}.$$
In fact by construction $A_{n,m}\in\mathcal{F}_{{\eta_{n,m}}^-}$ and therefore
$$0=\int_{A_{n,m}}U_{n,m}\,dP=\int_{A_{n,m}}\,(\Delta
M^{dp}_{\eta_{n,m}})^2\; \mathbb{I}_{\eta^{dp}_n=\eta_{n,m}}\,dP$$
so that $A_{n,m}\subset \tilde{A}_{n,m}$  up to a null measure set.~A similar procedure proves that $B_{l,h}\subset \tilde{B}_{l,h}$ up to a null measure set and we conclude that $[M,N]\equiv 0$  $P$-a.s.\bigskip\\
Viceversa, assume $[M,N]\equiv 0$  $P$-a.s.~This implies that $P$-a.s.~any addend in the sum of the right-hand side of (\ref{ref-rep-expl}) is null or equivalently  $P$-a.s.~for all $n,m$ and $l,h$ for any time $t$
$$\mathbb{I}_{\eta^{dp}_{n}=\eta_{n,m}}\;\mathbb{I}_{\tau^{dp}_{l}=\tau_{l,h}}\;\mathbb{I}_{\eta_{n,m}={\tau_{l,h}}}\;\mathbb{I}_{\eta_{n,m}\le t}=0$$
since, when $\eta_{n,m}$ and $\tau_{l,h}$ are finite,
$\Delta M^{dp}_{\eta_{n,m}}\Delta N^{dp}_{\tau_{l,h}}$ is $P$-a.s.~different from zero.\\
More briefly
\begin{equation}\label{from covariation null}\mathbb{I}_{\eta^{dp}_{n}=\eta_{n,m}}\;\mathbb{I}_{\tau^{dp}_{l}=\tau_{l,h}}\;\mathbb{I}_{\{\eta_{n,m}={\tau_{l,h}}<+\infty\}}=0\;\;\forall n,m,l,h,\;\;\;P\textrm{-a.s.}\end{equation}
In order to prove the singularity of the kernels we need to show that, if $P(\eta_{n,m}=\tau_{l,h}< +\infty)>0$, then (\ref{kernel singularity}) holds.~Actually recalling that $P^*\sim P$ it is enough to show that
 \begin{equation*}\mathbb{I}_{\{\eta_{n,m}=\tau_{l,h}< +\infty\}}U_{n,m}V_{l,h}=0, \ \ P^*\text{-a.s. }\end{equation*}
 or equivalently
  \begin{equation*}E^{P*}[\mathbb{I}_{\{\eta_{n,m}=\tau_{l,h}< +\infty\}} U_{n,m}V_{l,h}]=0.\end{equation*}
 Since (\ref{eq:P-marginals}) holds then
$$U_{n,m}=E^{P^*}\left[(\Delta
M^{dp}_{\eta_{n,m}})^2\; \mathbb{I}_{\eta^{dp}_n=\eta_{n,m}}
\mid\mathcal{F}_{{\eta_{n,m}}^-}\right],\;\;\;V_{l,h}=E^{P^*}\left[(\Delta
N^{dp}_{\tau_{l,h}})^2\; \mathbb{I}_{\tau^{dp}_l=\tau_{l,h}}
\mid\mathcal{H}_{{\tau_{l,h}}^-}\right]$$
and therefore
\begin{align*}&E^{P*}[\mathbb{I}_{\{\eta_{n,m}=\tau_{l,h}< +\infty\}} U_{n,m}V_{l,h}]=\cr&
E^{P*}\left[\mathbb{I}_{\{\eta_{n,m}=\tau_{l,h}< +\infty\}}E^{P^*}\left[(\Delta
M^{dp}_{\eta_{n,m}})^2\; \mathbb{I}_{\eta^{dp}_n=\eta_{n,m}}
\mid\mathcal{F}_{{\eta_{n,m}}^-}\right]E^{P^*}\left[(\Delta
N^{dp}_{\tau_{l,h}})^2\; \mathbb{I}_{\tau^{dp}_l=\tau_{l,h}}
\mid\mathcal{H}_{{\tau_{l,h}}^-}\right]\right].\end{align*}
Now applying Lemma 4.3 in \cite{caltor15} and observing that the random variable $\mathbb{I}_{\{\eta_{n,m}=\tau_{l,h}<+\infty\}}$ is $\mathcal{F}_{{\eta_{n,m}}^-}\vee\mathcal{H}_{{\tau_{l,h}}^-}$-measurable,  we get
\begin{align*}&E^{P*}[\mathbb{I}_{\{\eta_{n,m}=\tau_{l,h}< +\infty\}} U_{n,m}V_{l,h}]=\cr&
E^{P*}\left[\mathbb{I}_{\{\eta_{n,m}=\tau_{l,h}< +\infty\}}E^{P^*}\left[(\Delta
M^{dp}_{\eta_{n,m}})^2\; \mathbb{I}_{\eta^{dp}_n=\eta_{n,m}}
(\Delta
N^{dp}_{\tau_{l,h}})^2\; \mathbb{I}_{\tau^{dp}_l=\tau_{l,h}}
\mid\mathcal{F}_{{\eta_{n,m}}^-}\vee\mathcal{H}_{{\tau_{l,h}}^-}\right]\right]=\cr&
E^{P*}\left[E^{P^*}\left[\mathbb{I}_{\{\eta_{n,m}=\tau_{l,h}< +\infty\}}(\Delta
M^{dp}_{\eta_{n,m}})^2\; \mathbb{I}_{\eta^{dp}_n=\eta_{n,m}}
(\Delta
N^{dp}_{\tau_{l,h}})^2\; \mathbb{I}_{\tau^{dp}_l=\tau_{l,h}}
\mid\mathcal{F}_{{\eta_{n,m}}^-}\vee\mathcal{H}_{{\tau_{l,h}}^-}\right]\right]
=\cr&
E^{P^*}\left[(\Delta M^{dp}_{\eta_{n,m}})^2\; (\Delta
N^{dp}_{\tau_{l,h}})^2\;\mathbb{I}_{\{\eta_{n,m}=\tau_{l,h}< +\infty\}}
\mathbb{I}_{\eta^{dp}_n=\eta_{n,m}}
 \mathbb{I}_{\tau^{dp}_l=\tau_{l,h}}\right]=0,
\end{align*}
where last equality follows by (\ref{from covariation null}).
\end{proof}
Next result states the equivalence between the mutual behavior of the sharp brackets of $M$
and $N$ and the existence of a  $\mathbb{G}$-basis
of dimension one.~The starting point is Theorem 4.5 in \cite{caltor15}, which requires both the $P$-triviality of the $\sigma$-algebras $\mathcal{F}_0$ and $\mathcal{H}_0$ and the $(P,\mathbb{F})$-PRP of $M$ and the $(P,\mathbb{H})$-PRP of $N$.\bigskip\\
\textit{From now on, by virtue of Theorem \ref{thm-2assetpricing}, we will assume the following condition \textrm{(A1)}}.\bigskip\\
 \textrm{(A1)}\;
 \textit{$\mathbb{P}(M,\mathbb{F})=\{P|_{\mathcal{F}_T}\}$} and
 \textit{$\mathbb{P}(N,\mathbb{H})=\{P|_{\mathcal{H}_T}\}$}.
\begin{proposition}\label{lemma: M+N basis}
Let $P^*$ be the probability
measure defined by (\ref{eq-mgpreserving-decoupl-measure}).
\phantom{pro}
\begin{itemize}
    \item [(i)] If the kernels $d \langle M\rangle^{P,\mathbb{F}}$ and $d \langle
N\rangle^{P,\mathbb{H}}$ are mutually singular, then
$M+N$ enjoys the $(P^*,\mathbb{G})$-PRP.
    \item [(ii)] If there exists $Z\in\mathcal{M}^2(P^*,\mathbb{G})$ which enjoys the $(P^*,\mathbb{G})$-PRP, then the kernels $d \langle M\rangle^{P,\mathbb{F}}$ and $d \langle
N\rangle^{P,\mathbb{H}}$ are mutually singular.
\end{itemize}
\end{proposition}
\begin{proof}
(i) Under $P^*$ the triplet $(M,N,[M,N])$  is a
$(P^*,\mathbb{G})$-basis (see Theorem 4.5 in \cite{caltor15}).~If
 $d\langle M\rangle^{P,\mathbb{F}}$ and $d\langle
N\rangle^{P,\mathbb{H}}$ are mutually singular, then the
same holds for  $d\langle M^{dp}\rangle^{P,\mathbb{F}}$ and
$d\langle N^{ dp}\rangle^{P,\mathbb{H}}$.~By Lemma \ref{lemma-cns-cov0}  the covariation process
$[M,N]$ is $P$-a.s.~null, so that the $(P^*,\mathbb{G})$-basis
reduces to the pair $(M,N)$.\bigskip\\
As a consequence, if $K\in
L^2(\Omega, \mathcal{G}_T, P^*)$, then
\begin{equation}\label{eq-M,N-representation}
K=k_0+\int_0^T\gamma_sdM_s+\int_0^T\eta_sdN_s,
\end{equation}
where $k_0$ is a constant, $(\gamma)_{t\in[0,T]}$ and $(\eta)_{t\in[0,T]}$ are
$\mathbb{G}$-predictable processes satisfying
$E^{P^*}\left[\int_0^T\gamma^2_t\,d[M]_t\right]<+\infty$ and
$E^{P^*}\left[\int_0^T\eta^2_t\,d[N]_t\right]<+\infty$, respectively.\bigskip\\
Moreover, since $P^*|_{\mathcal{F}_T}$  coincides with
$P|_{\mathcal{F}_T}$ then on one side $\langle
M\rangle^{P,\mathbb{F}}=\langle M\rangle^{P^*,\mathbb{F}}$ and,
since $P^*$ decouples $\mathbb{F}$ and $\mathbb{H}$, on the other
side $\langle M\rangle^{P^*,\mathbb{F}}=\langle
M\rangle^{P^*,\mathbb{G}}$ and therefore
$$
\langle M\rangle^{P,\mathbb{F}}=\langle M\rangle^{P^*,\mathbb{G}}.
$$
Analogously  $\langle N\rangle^{P,\mathbb{H}}=\langle
N\rangle^{P^*,\mathbb{G}}$.~Then by Proposition
\ref{lemma-disjoint-pred-supp} there exist  two $\mathbb{G}$-predictable subsets of $\Omega\times[0,T]$, $C^M$ and $C^N$, with $C^N=
\overline{C^M}$, such that
  $$\int_{C^N}d\langle M\rangle^{P^*,\mathbb{G}}_t=0,\;\;\;\;\;\;\;
 \int_{C^M}d\langle N\rangle^{P^*,\mathbb{G}}_t=0.$$
 Let
$(\lambda)_{t\in[0,T]}$ be defined at time $t$ by
$$
\lambda_t=\gamma_t\mathbb{I}_{C^M}(t)+\eta_t\mathbb{I}_{C^N}(t).
$$
Since the indicator functions $\left(\mathbb{I}_{C^M}(t)\right)_{t\in[0,T]}$ and
$\left(\mathbb{I}_{C^N}(t)\right)_{t\in[0,T]}$ of the predictable sets $C^M$ and $C^N$, respectively, are
predictable processes, $\lambda$ is a $\mathbb{G}$-predictable process.\bigskip\\
Moreover
$E^{P^*}\left[\int_0^T\lambda^2_t\,d[M+N]_t\right]<+\infty$ and
the following equalities hold
$$E^{P^*}\left[\int_0^T\left(\lambda_s-\gamma_s\right)^2d\langle
M\rangle^{P^*,\mathbb{G}}_s\right]=E^{P^*}\left[\int_0^T\left(\lambda_s-\eta_s\right)^2d\langle
N\rangle^{P^*,\mathbb{G}}_s\right]=0.$$
As a consequence we get
$$
\int_0^T\gamma_sdM_s=\int_0^T\lambda_sdM_s,\;\;\;\;\;\;
\int_0^T\eta_sdN_s=\int_0^T\lambda_sdN_s,
$$
so that (\ref{eq-M,N-representation}) can be rewritten as
\begin{equation*}
K=k_0+\int_0^T\lambda_sd(M_s+N_s)
\end{equation*}
and by the arbitrariness of $K$ we derive that $M+N$ enjoys the
$(P^*,\mathbb{G})$-PRP.\bigskip\\
(ii) Let $Z$ enjoy the $(P^*,\mathbb{G})$-PRP.~Then
\begin{equation*}
M_t=\int_0^t\alpha_sdZ_s, \ \ \ \ N_t=\int_0^t\beta_sdZ_s,
\end{equation*}
with $(\alpha)_{t\in[0,T]}$ and $(\beta)_{t\in[0,T]}$
$\mathbb{G}$-predictable processes which satisfy
$E^{P^*}\left[\int_0^T\alpha^2_t\,d[Z]_t\right]<+\infty$ and
$E^{P^*}\left[\int_0^T\beta^2_t\,d[Z]_t\right]<+\infty$, respectively.
\bigskip\\
By construction $M$ and $N$ are independent
$(P^*,\mathbb{G})$-martingales, so that
$$
0=\langle M,
N\rangle^{P^*,\mathbb{G}}_\cdot=\int_0^\cdot\alpha_s\beta_sd\langle
Z\rangle^{P^*,\mathbb{G}}_s,
$$
that is \begin{align}\label{eq:alfabeta}\alpha_\cdot\beta_\cdot=0,\;\;\;\;d\langle
Z\rangle^{P^*,\mathbb{G}}_\cdot\text{-}a.s.\end{align}
Observe that
$$ \langle
M\rangle^{P,\mathbb{F}}_\cdot=\langle
M\rangle^{P^*,\mathbb{G}}_\cdot=\int_0^\cdot\alpha^2_sd\langle
Z\rangle^{P^*,\mathbb{G}}_s,
$$
and
$$\langle
 N\rangle^{P,\mathbb{H}}_\cdot=\langle
 N\rangle^{P^*,\mathbb{G}}_\cdot=\int_0^\cdot\beta^2_sd\langle
Z\rangle^{P^*,\mathbb{G}}_s.$$
In order to prove that  the kernels $d \langle M\rangle^{P,\mathbb{F}}$ and $d \langle
N\rangle^{P,\mathbb{H}}$ are mutually
singular, let us introduce the predictable set
$$\Gamma:=\{(\omega,t), \alpha^2_t(\omega)>0\}.$$
Then $$\int_{\overline{\Gamma^\omega}}\alpha^2_s(\omega)d\langle
Z\rangle^{P^*,\mathbb{G}}_s(\omega)=0,$$
or equivalently
$$ d\langle M\rangle^{P,\mathbb{F}}_\cdot(\omega)(\overline{\Gamma^\omega})=0,$$
that is the $\omega$-section of $\Gamma$ is a support of the measure $d\langle M\rangle^{P,\mathbb{F}}_\cdot(\omega)$.~Moreover
\begin{align*}&\int_{\Gamma^\omega}\beta^2_s(\omega)d\langle
Z\rangle^{P^*,\mathbb{G}}_s(\omega)=\cr&\int_{\{t,\, \alpha_t(\omega)\neq 0\}}\beta^2_s(\omega)d\langle
Z\rangle^{P^*,\mathbb{G}}_s(\omega)=\cr&\int_{[0,T]}\mathds{1}_{\{t,\, \alpha_t(\omega)\neq 0\}}(s)\beta^2_s(\omega)d\langle
Z\rangle^{P^*,\mathbb{G}}_s(\omega)=\cr&\int_{[0,T]}\mathds{1}_{\{t,\, \alpha_t(\omega)\neq 0\}}(s)\mathds{1}_{\{t,\, \beta_t(\omega)\neq 0\}}(s)\beta^2_s(\omega)d\langle
Z\rangle^{P^*,\mathbb{G}}_s(\omega)=\cr&\int_{[0,T]}\mathds{1}(s)_{\{t,\, (\alpha_t\beta_t)(\omega)\neq 0\}}\beta^2_s(\omega)d\langle
Z\rangle^{P^*,\mathbb{G}}_s(\omega)=0,\end{align*}
where the last equality follows by (\ref{eq:alfabeta}).~Therefore
$$ d\langle N\rangle^{P,\mathbb{H}}_\cdot(\omega)(\Gamma^\omega)=0,$$
 that is
 the $\omega$-section of $\overline{\Gamma}$ is a support of the measure $d\langle N\rangle^{P,\mathbb{H}}_\cdot(\omega)$ and we get proof.
\end{proof}
Finally we state the main result of this section.~For the
sake of clarity we stress that $\langle
M\rangle^{P^*,\mathbb{F}}=\langle M\rangle^{P,\mathbb{F}}$ and
$\langle N\rangle^{P^*,\mathbb{H}}=\langle
N\rangle^{P,\mathbb{H}}$.
\begin{theorem}\label{thm-singuar-brackets}
The dimension of $\mathcal{H}^1(P,\mathbb{G})$ is at most three and in particular is
\begin{itemize}
\item [(i)]  equal to one if and only if  the kernels $d\langle M\rangle^{P,\mathbb{F}}$ and $d\langle N\rangle^{P,\mathbb{H}}$ are mutually singular;
\item [(ii)] equal to two if  the kernels
 $d\langle M\rangle^{P,\mathbb{F}}$ and $d\langle N\rangle^{P,\mathbb{H}}$ are not mutually singular but the kernels $d\langle M^{dp}\rangle^{P,\mathbb{F}}$ and $d\langle N^{dp}\rangle^{P,\mathbb{H}}$
are mutually singular.
\end{itemize}
\end{theorem}
\begin{proof}
$dim\big(\mathcal{H}^1(P,\mathbb{G})\big)$ coincides with $dim\big(\mathcal{H}^1(P^*,\mathbb{G})\big)$ by Proposition \ref{thm-dimension}.~By Theorem 4.5 in \cite{caltor15} the set $\mathbb{P}\big((M,N,[M,N]),\mathbb{G}\big)$ is a singleton, that is $(M,N,[M,N])$ enjoys the $(P^*,\mathbb{G})$-PRP and therefore, by the remark following Definition \ref{def_dimension}, one gets $dim\big(\mathcal{H}^1(P^*,\mathbb{G})\big)$ is less or equal to three.\\
(i) This point immediately follows by Proposition \ref{lemma: M+N basis}.\\
(ii) If $dim\big(\mathcal{H}^1(P^*,\mathbb{G})\big)$ is greater than one previous
point implies that the kernels $d\langle M\rangle^{P,\mathbb{F}}$ and $ d\langle
N\rangle^{P,\mathbb{H}}$ are not mutually singular.~Moreover if the kernels
$d\langle M^{dp}\rangle^{P,\mathbb{F}}$ and $d\langle
N^{dp}\rangle^{P,\mathbb{H}}$ are  mutually singular, as
stated by Lemma \ref{lemma-cns-cov0}, $[M,N]=0$ and $\mathbb{P}\big((M,N),\mathbb{G}\big)=\{P^*|_{\mathbb{G}}\}$, so that $dim\big(\mathcal{H}^1(P^*,\mathbb{G})\big)$ is equal to two.
\end{proof}
\begin{remark}\label{rem-ese-jean}
Proposition \ref{lemma: M+N basis} and Theorem
\ref{thm-singuar-brackets} have been inspired by Theorem 9.5.2.4, page 540,
Theorem 9.5.2.5, page 541,  and subsequent arguments in
\cite{Jean-yor-chesney} and in particular by the following
example.~A Brownian motion
 $W$
 and a compensated Poisson martingale $\Pi$ are considered.~If $f$ and $g$ are
 deterministic functions such that $fg=0$ then
 the process $dX_t=f(t)\,dW_t+g(t)\,d\Pi_t$  enjoys the PRP with
 respect to its natural filtration.
\end{remark}
\section{$(P,\mathbb{G})$-martingale representation}
In this section, in the same setting of last two sections and in particular under the standing assumptions \textrm{(D)} and \textrm{(A1)}, we present our martingale representation result on $\mathbb{G}$ under $P$.~First we deal with the case when $\mathbb{G}$ is the progressive enlargement of $\mathbb{F}$ by a general random time $\tau$ and then we consider the case when $\mathbb{G}$ is the enlargement of $\mathbb{F}$ by the reference  filtration $\mathbb{H}$ of a martingale enjoying the PRP.\bigskip\\
In view of possible applications to financial models and
although this is not essential for the validity of the result, we assume that $M$ is the martingale part of a square-integrable special semi-martingale  $X=(X_t)_{t\in [0,T]}$ on $(\Omega,\mathcal{F},P)$
 with canonical decomposition
 \begin{equation}\label{semi-martingale}
 X_t=X_0+M_t+A_t,
 \end{equation}
such that
 \begin{equation*}
 E^P\left[X_0^2+[ M]_T+|A|^2_T\right]<+\infty.
 \end{equation*}
As usual $A$ is a
predictable process of finite variation, $M_0=A_0=0$ and $|A|$
denotes the total variation process of $A$ (see
 VII-98, page 294, in \cite{del-me-b}).~In particular we assume that there
 exists a predictable process $\alpha=(\alpha_t)_{t\in [0,T]}$ $P$-a.s.~in the
 space $L^2_{loc}([0,T],\mathcal{B}([0,T]),d\langle M\rangle^{P,\mathbb{F}})$ such that
\begin{equation}\label{structure condition}A_t=\int_0^t\alpha_s\,d\langle M\rangle^{P,\mathbb{F}}_s.\end{equation}
\begin{remark}
We note that if
$\mathbb{P}\left(X,\mathbb{F}\right)$ is a singleton then condition \textrm{(A1)} for $M$ holds,
that is $\mathbb{P}(M,\mathbb{F})=\{P|_{\mathcal{F}_T}\}$ (see Proposition 2.2 in \cite{caltor18}).~Moreover
if $\mathbb{P}\left(X,\mathbb{F}\right)=\{P^X\}$ and the derivative $dP^X/dP|_{\mathcal{F}_T}$ is locally square-integrable,
then (\ref{structure condition}) holds (see Proposition 4 in \cite{schweizer92}).
\end{remark}
\begin{proposition}
$X$ and $M$ share the same jump times.
\end{proposition}
\begin{proof}
By Yoeurp's decomposition $\langle M\rangle^{P,\mathbb{F}}$ coincides with the sum of
$\langle M^c\rangle^{P,\mathbb{F}}$, $\langle M^{dq}\rangle^{P,\mathbb{F}}$ and $\langle M^{dp}\rangle^{P,\mathbb{F}}$.~The first two processes
are continuous and, by Proposition \ref{cor-explicit-compensator}, $\langle M^{dp}\rangle^{P,\mathbb{F}}$ is a pure jump process which
shares its jump times with $M^{dp}$.~It follows that the jump times of $A$ are a subset of those of $M$ and the thesis follows
by equality (\ref{semi-martingale}).
\end{proof}
\subsection{Progressive enlargement by a random time}
Let $\tau$ be any random time in $[0,+\infty]$.~For the sake of notational convenience we will rename by $H$
 the compensated occurrence process of $\tau$, $H^{\tau}$, defined as in Subsection \ref{subsec-compensated op} and by  $\mathbb{H}$ its natural filtration.~By Proposition \ref{prop-regularity} $H$ is square-integrable and by Theorem \ref{prop-H-prp} it
enjoys the $(P,\mathbb{H})$-PRP.~Here $H$  plays the role of $N$ and the filtration $\mathbb{G}$ defined in (\ref{def-G})
is the progressive enlargement of
$\mathbb{F}$ by $\tau$.~Condition $\textrm{(D)}$ is assumed.
\bigskip\\
According to  (\ref{eq-eta-dec}) we denote by $\tau^{dp}$ and $\tau^{dq}$  the $\mathbb{H}$-accessible and the $\mathbb{H}$-totally inaccessible
component of $\tau$, respectively.~We stress that Proposition 3.1 applies to prove that $\tau^{dp}$ and $\tau^{dq}$ are also
the $\mathbb{G}$-accessible and the $\mathbb{G}$-totally inaccessible component of $\tau$,
respectively.~So, we will simply refer to $\tau^{dp}$ ($\tau^{dq}$) as to the accessible (totally inaccessible) component of $\tau$ and, as usual,
$(\tau^{dp}_h)_{h\in \mathbb{N}}$ will denote an enveloping sequence of $\tau^{dp}$.
\begin{proposition}\label{prop-repres1}
\begin{itemize}
\item[(i)] The
triplet $(M,H, [M,H])$ enjoys the $(P^*,\mathbb{G})$-PRP and is a $(P^*,\mathbb{G})$-basis of
martingales.
  \item[(ii)] $[M,H]$ admits the representation
\begin{equation}\label{eq: explicit covariation}
    [M,H]_t=\sum_{h\in \mathbb{N}}\Delta M_{\tau^{dp}_h}\left(\mathbb{ I}_{\{\tau^{dp}=\tau^{dp}_h\}}-
    P(\tau^{dp}=\tau^{dp}_h\mid\mathcal{H}_{{\tau^{dp}_h}^-})\right)\mathbb{ I}_{\{\tau^{dp}_h\le
    t\}}.
\end{equation}
\end{itemize}
\end{proposition}
\begin{proof}
The $(P^*,\mathbb{G})$-martingales $M$ and $H$ are independent and therefore strongly
 orthogonal
 so that point (i) follows by Theorem 4.5 in
\cite{caltor15}.\\
As far as point (ii) is concerned, by Proposition \ref{overlap} it
is enough to compute $[M^{dp},H^{dp}]$, where $H^{dp}$ denotes the accessible part of the Yoeurp decomposition of $H$ (see (\ref{Yoeurp-decomposition})).~Then representation (\ref{eq: explicit covariation}) follows immediately recalling representation (\ref{H-prime}) and equality (\ref{ref-rep}).
\end{proof}
Point (i) above joint with Proposition \ref{lemma:invariance} assures that
 a triplet of local martingales driving
the $(P,\mathbb{G})$-representation exists.~Nevertheless, in order to give an explicit description of its components, we need some preliminaries results.\bigskip\\
Let $A^{\tau,P, \mathbb{G}}$ be the $(P,\mathbb{G})$-compensator of $\tau$.~Then
$H^\prime=(H^\prime_t)_{t\in [0,T]}$ defined by
$$H^\prime_t:=\mathbb{ I}_{\{\tau\le t\}}-A^{\tau,P,\mathbb{G}}_t$$ is a $(P,\mathbb{G})$-martingale.~We may refer to $H^\prime$ as to the
\textit{$(P,\mathbb{G})$-compensated occurrence process of $\tau$}.
\begin{proposition}\label{prop-G-compensator}
\phantom{bla}
\begin{itemize}
\item[(i)] $H^\prime\in\mathcal{M}^2(P,\mathbb{G})$ and satisfies
\begin{equation*}
H^\prime_t=\mathbb{I}_{\{\tau\le t\}}-
A^{dq,P,\mathbb{G}}_t-\sum_{h\in \mathbb{N}}P(\tau^{dp}=\tau^{dp}_h\mid\mathcal{G}_{{{\tau^{dp}_h}^-}})\mathbb{ I}_{\{\tau^{dp}_h\le t\}},
\end{equation*}
where $A^{dq,P,\mathbb{G}}$ denotes the $(P,\mathbb{G})$-compensator of $\tau^{dq}$.
  \item[(ii)] $H^\prime$ is a special $(P^*,\mathbb{G})$-semi-martingale and
 $H$  is its martingale part.
\end{itemize}
\end{proposition}
\begin{proof}
Let $A^{dp,P,\mathbb{G}}$ denote the $(P,\mathbb{G})$-compensator of $\tau^{dp}$.~In order to show point (i) observe that Proposition \ref{prop-regularity} applies to prove that $H^\prime$ is
square-integrable and the same procedure used to get (\ref{eq-dec-comp-H}) yields
\begin{equation}\label{eq-comp1}A^{\tau,P,\mathbb{G}}=A^{dp,P,\mathbb{G}} + A^{dq,P,\mathbb{G}}.\end{equation}
The process $A^{dq,P,\mathbb{G}}$ is a continuous process, while
  by Lemma \ref{prop-explicit-compensator}  $A^{dp,P,\mathbb{G}}$ obeys the equality
\begin{equation}\label{eq-comp2}
A^{dp,P,\mathbb{G}}_t=\sum_{h\in \mathbb{N}}P(\tau^{dp}=\tau^{dp}_h\mid\mathcal{G}_{{\tau^{dp}_h}^-})\mathbb{
I}_{\{\tau^{dp}_h\le
  t\}}.
\end{equation}
Point (ii) follows immediately by the equality
 \begin{equation}\label{eq-mart-part}
H^\prime_t=H_t+A_t^{\tau,P,\mathbb{H}}-A_t^{\tau,P,\mathbb{G}}
\end{equation}
considering that by construction $H$ is a $(P^*,\mathbb{G})$-martingale under $P^*$, as well as under $P$,
and the process $A^{\tau,P,\mathbb{H}}-A^{\tau,P,\mathbb{G}}$
is $\mathbb{G}$-predictable and of finite variation.
\end{proof}
\begin{remark}\label{rem-tau e tau_dp}
Since $\mathbb{I}_{\{\tau^{dp}_h\le
  t\}}$ is $\mathcal{G}_{{\tau^{dp}_h}^-}$-measurable and $\{\tau^{dp}<+\infty\}\subset\{\tau=\tau^{dp}\}$,
   (\ref{eq-comp2}) can be rewritten as
\begin{equation*}
A^{dp,P,\mathbb{G}}_t=\sum_{h\in \mathbb{N}}P(\tau=\tau^{dp}_h\mid\mathcal{G}_{{\tau^{dp}_h}^-})\mathbb{
I}_{\{\tau^{dp}_h\le
  t\}}.
\end{equation*}
\end{remark}
According to
Definition \ref{def-min-mart-meas} we introduce a new
 condition.\vspace{3mm}\\
\textrm{(A2)} $P$ is the
minimal martingale measure for $H^\prime$ under $P^*$.\vspace{3mm}\\
Let us analyze condition \textrm{(A2)} and in particular its relationship with the \textit{immersion property} $\mathbb{F}\underset{P}{\hookrightarrow} \mathbb{G}$, that is the inclusion $\mathcal{M}_{loc}(P,\mathbb{F})\subset \mathcal{M}_{loc}(P,\mathbb{G})$ (see Section 5.9.1, page 315, in \cite{Jean-yor-chesney}).\bigskip\\
It is to stress that under \textrm{(A2)} the process $L$ defined by $L_t := \frac{dP}{dP^*}\Big|_{\mathcal{G}_t}$ satisfies
\begin{equation}\label{eq-L_rappr_diff}
  L_t=1-\int_0^t\gamma_sL_{s^-}dH_s
\end{equation}
where $\gamma$ is a suitable $\mathbb{G}$-predictable process (see
Remark 2 in \cite{schweizer92} or Theorem 9 in
\cite{ans_str92}).\\
\begin{proposition}\label{prop: condition A2)}
\phantom{bla}
\begin{itemize}
  \item [(i)] \textrm{(A2)} holds if and only if $M$ and $[M,H]$ are $(P,\mathbb{G})$-local martingales.
    \item [(ii)]  \textrm{(A2)} implies
$\mathbb{F}\underset{P}{\hookrightarrow} \mathbb{G}$.
      \item [(iii)] When $\tau$ is totally inaccessible, \textrm{(A2)} is equivalent to $\mathbb{F}\underset{P}{\hookrightarrow} \mathbb{G}$.
\end{itemize}
\end{proposition}
\begin{proof}
(i)
Assuming \textrm{(A2)} by point (i) of Proposition \ref{prop-repres1} it follows immediately that $M$ and $[M,H]$ are $(P,\mathbb{G})$-local martingales.~More precisely, it is easy to check that $M$ and $[M,H]$ belong to $\mathcal{M}^2(P,\mathbb{G})$.~Viceversa, let $V$ be any element of $\mathcal{M}^2_{loc}(P^*,\mathbb{G})$ which is orthogonal to the martingale part $H$ of the special $(P^*,\mathbb{G})$-martingale $H^\prime$.~We have to show that $V$ belongs to $\mathcal{M}(P,\mathbb{G})$.~This is enough since
$P|\mathcal{G}_0=P^*|\mathcal{G}_0$.~By part i) of Proposition \ref{prop-repres1} there exist two predictable processes $\alpha$ and $\beta$ such that
$$V_t=\int_0^t\alpha_s\;dM_s +\int_0^t\beta_s\;d[M,H]_s.$$
Since equation (\ref{eq-L_rappr_diff}) holds, the processes   $LM$ e $L[M,H]$ are $(P^*,\mathbb{G})$-local martingales and therefore also $[L,M]$ e $[L,[M,H]]$ are $(P^*,\mathbb{G})$-local martingales (see Lemma 15.2.1, page 373, in \cite{coh-ell-15}).~Moreover
 \begin{align*} L_t\,V_t=&\int_0^t\,L_{s^-}\,dV_s+\int_0^t\,V_{s^-}\,dL_s+[L,V]_t\cr=&\int_0^t\,L_{s^-}\,dV_s+\int_0^t\,V_{s^-}\,dL_s +\int_0^t\,\alpha_s\,d[L,M]_s+\int_0^t\,\beta_s\,d[L,[M,H]]_s,\end{align*}
 so that $LV$ is a $(P^*,\mathbb{G})$-local martingale and therefore $V$ is a $(P,\mathbb{G})$-local martingale.\bigskip\\
(ii) This point follows immediately considering that, since $M$ enjoys the $(P,\mathbb{F})$-PRP, any $(P,\mathbb{F})$-local martingale can be represented as integral w.r.t.~$M$, which by assumption is a $(P,\mathbb{G})$-martingale.\bigskip\\
(iii) By Proposition \ref{overlap}  $\tau$ totally
inaccessible implies that $[M,H]\equiv 0$ and $\mathbb{F}\underset{P}{\hookrightarrow} \mathbb{G}$ implies that $M$ is a $(P,\mathbb{G})$-local martingale, so that by point (i) \textrm{(A2)} holds.~The proof ends using point (ii).
\end{proof}
We state now our main result in the framework of progressive enlargement completed with all the assumptions.
\begin{theorem}\label{prop-repres2}
If conditions  \textrm{(D)}, \textrm{(A1)} and \textrm{(A2)} hold, then
the triplet $\left(M, H^\prime,[M,H]\right)$ enjoys the
$(P,\mathbb{G})$-PRP.
\end{theorem}
\begin{proof}
By point (i) of Proposition \ref{prop-repres1} and Proposition \ref{lemma:invariance} we derive the
$(P,\mathbb{G})$-PRP for the
 triplet $(\tilde{M}, \tilde{H}, \tilde{K})$ defined by
\begin{align}\label{eq-triplet-martingale}
\tilde{M}_t:=&M_t-\int_0^t\frac{1}{L_{s^-}}d\langle L, M\rangle^{P^*,\mathbb{G}}_s, \notag \\
    \tilde{H}_t:=&H_t-\int_0^t\frac{1}{L_{s^-}}d\langle L,
    H\rangle^{P^*,\mathbb{G}}_s,\\
    \tilde{K}_t:=&[M,H]_t-\int_0^t\frac{1}{L_{s^-}}d\langle L,
    [M,H]\rangle^{P^*,\mathbb{G}}_s,\notag
\end{align}
provided the sharp brackets $\langle L,
M\rangle^{P^*,\mathbb{G}}$, $\langle L, H\rangle^{P^*,\mathbb{G}}$
and $\langle L,
    [M,H]\rangle^{P^*,\mathbb{G}}$ exist.~Since equation (\ref{eq-L_rappr_diff}) implies that $LM$ and $L[M,H]$ are  $(P^*,\mathbb{G})$-local martingales then
    $$\langle L, M\rangle^{P^*,\mathbb{G}}=\langle L,
    [M,H]\rangle^{P^*,\mathbb{G}}\equiv 0$$ (see Lemma 15.2.1, page 373, in \cite{coh-ell-15}).~Moreover, equation (\ref{eq-mart-part}) implies that $H$ is a special
$(P,\mathbb{G})$-semi-martingale and therefore the existence of
$\langle L,
    H\rangle^{P^*,\mathbb{G}}$ follows
    by Lemma 2.1 and subsequent Remark 2.1 in \cite{caltor18}.\\
    Then by the equations (\ref{eq-triplet-martingale}) we derive immediately that $\tilde{M}=M$ and
$\tilde{K}=[M,H]$ and also that $\tilde{H}=H^\prime$ considering that
$$\tilde{H}_t-H^\prime_t=A^{\tau,P,\mathbb{G}}_t-A^{\tau,P,\mathbb{H}}_t-\int_0^t\frac{1}{L_{s^-}}d\langle
L,
    H\rangle^{P^*,\mathbb{G}}_s,$$
and therefore $\tilde{H}-H^\prime$ turns out to be null since it is a predictable $(P,\mathbb{G})$-martingale of finite variation (see Lemma 22.11, page 416,
in \cite {kall97}).
\end{proof}
\begin{corollary} \label{cor: tau totally in} Assume \textrm{(D)}, \textrm{(A1)} and \textrm{(A2)}.~If $\tau$ is totally
inaccessible, then the pair $(M,H^\prime)$ is a $(P,\mathbb{G})$-basis.
\end{corollary}
\begin{proof}
Proposition \ref{overlap} applies to prove that $[M,H]$ is identically zero.~Moreover, under the assumptions of Theorem \ref{prop-repres2}, the square-integrable $(P,\mathbb{G})$-martingale $H^\prime$ is strongly orthogonal to $M$.~In fact it is enough to observe that (\ref{eq-mart-part}) implies
\begin{align}\label{eq-equivalence-i-iii}
[M,H']=[M,H]+[M,A^{\tau,P,\mathbb{H}}-A^{\tau,P,\mathbb{G}}]
\end{align}
and that by Lemma 2.3 in \cite{yoeurp76} the process $[M,A^{\tau,P,\mathbb{H}}-A^{\tau,P,\mathbb{G}}]$ is a
$(P,\mathbb{G})$-local martingale.
\end{proof}
When $[M,H]$ is not identically zero, the triplet $\left(M, H^\prime,[M,H]\right)$  is
not a $(P,\mathbb{G})$-basis.~In fact the $(P,\mathbb{G})$-martingales $M$ and $[M,H]$ are not strongly orthogonal
 and the same happens for $H^\prime$ and $[M,H]$.~Let us briefly discuss this point.~\\
 $M$  and $[M,H]$  are
strongly orthogonal $(P,\mathbb{G})$-martingales if and only if
  $M[M,H]$ is a $(P,\mathbb{G})$-martingale and the latter holds  if and
only if $LM[M,H]$ is a $(P^*,\mathbb{G})$-local martingale.~Moreover
$$
L_tM_t=\int_0^tL_{s^-}dM_s+\int_0^tM_{s^-}\gamma_sL_{s^-}dH_s+\int_0^t\gamma_sL_{s^-}d[M,H]_s,
$$
so that
\begin{align*}
\big[LM,[M,H]\big]_t=&\int_0^tL_{s^-}d\big[M,[M,H]\big]_s+\int_0^tM_{s^-}\gamma_sL_{s^-}d\big[H,[M,H]\big]_s+\\
+&\int_0^t\gamma_sL_{s^-}d\big[[M,H],[M,H]\big]_s.
\end{align*}
The right hand side of previous equality is a $(P^*,\mathbb{G})$-local martingale if and only if $[M,H]$ is
identically zero.~Similarly we derive that $H^\prime$ and $[M,H]$ are not strongly orthogonal.
\begin{remark}\label{rem-immersion}
Theorem \ref{prop-repres2} generalizes Proposition 5.3 (ii) in \cite{ca-jean-za13}.~In that paper, in addition to \textrm{(A1)}, the authors assume
$\mathbb{F}\underset{P}{\hookrightarrow}
\mathbb{G}$ and \textit{Jacod's equivalence hypothesis for $\tau$} w.r.t.~a
non-atomic measure $\nu$  (see Condition (A) and Proposition 1.5 in \cite{jacod85}).~Last condition implies the \textit{density
hypothesis for $\tau$} (see e.g.~\cite{elka-jean-jiao09})  and as a consequence its total inaccessibility (see Remark \ref{caratterizzazione dei tempi aleatori}).~Corollary 2.8  in \cite{ame-be-schw03} assures condition \textrm{(D)}.~Then by point (iii) of Proposition \ref{prop: condition A2)} it follows that \textrm{(A2)} holds.~So, when the processes involved live on $[0,T]$, our assumptions are satisfied and, by Corollary \ref{cor: tau totally in}, $(M,H^\prime)$ is a $(P,\mathbb{G})$-basis.\\
For the sake of completeness we note that in general, under $\mathbb{F}\underset{P}{\hookrightarrow}
\mathbb{G}$, if
\begin{equation*}\int_0^\cdot \frac{dP[\tau\le u| \mathcal{F}_u]}{1-P[\tau < u|
\mathcal{F}_u]},\end{equation*} is a $\mathbb{F}$-predictable process, then its value at $t\wedge \tau$ coincides $P$-a.s.~with
$A^{\tau,P,\mathbb{G}}_t$ (see Proposition 6.3 and previous comment in \cite{jean-rut-99}).~This happens for example under the density
hypothesis for $\tau$, so that
\begin{equation*}
A^{\tau,P,\mathbb{G}}_t=\int_0^{\tau\wedge t}\lambda_u\,\nu(du),
\end{equation*}
 where the \textit{intensity} $\lambda$ is defined by
\begin{equation*}\label{intensity}\lambda_t:=\frac{p_t(t)}{P(\tau>t\mid\mathcal{F}_t)},\end{equation*}
with $\left(p_t(u)\right)_{u\in [0,T]}$ the $\mathbb{F}$-\textit{conditional density of $\tau$}, that is
$$P(\tau>t\mid\mathcal{F}_t)=\int_t^{+\infty}p_t(u)\,\nu(du)$$
(see formulas (3) and (5) in \cite{ca-jean-za13}).~So equality (\ref{H-prime}) reduces to
\begin{equation*}\label{the basis}H^\prime_t=\mathbb{I}_{\{\tau\leq t\}}-\int_0^{\tau\wedge t}\lambda_u\,\nu(du).\end{equation*}
\end{remark}
\begin{remark}\label{rem-no decoupling}
It may be worthwhile to observe that, dropping condition $\textrm{(D)}$, which is a standing hypothesis in this paper, and denoting by
 $\tau^{dp}$ and $\tau^{dq}$  the $\mathbb{G}$-accessible and the $\mathbb{G}$-totally inaccessible
component of $\tau$, respectively (possibly different from the $\mathbb{F}$-accessible and the $\mathbb{F}$-totally inaccessible
component of $\tau$), the expression of the $\mathbb{G}$-compensator of $\tau$ given by (\ref{eq-comp1}) and
(\ref{eq-comp2}) still holds.\\In this more general framework if,
for any $h$, $\tau^{dp}_h$, is an
$\mathbb{F}$-stopping time then $\mathcal{F}_{{\tau^{dp}_h}^-}$ is well-defined and, as it can be easily proved,
\begin{equation*}\mathcal{G}_{{\tau^{dp}_h}^-}\cap\{\tau=\tau^{dp}_h\}=\mathcal{F}_{{\tau^{dp}_h}^-}\cap\{\tau=\tau^{dp}_h\}\end{equation*}
 and
\begin{equation*}
P(\tau=\tau^{dp}_h\mid\mathcal{G}_{{\tau^{dp}_h}^-})\mathbb{
I}_{\{\tau^{dp}_h\le
  t\}}=P(\tau=\tau^{dp}_h\mid\mathcal{F}_{{\tau^{dp}_h}^-})\mathbb{
I}_{\{\tau^{dp}_h\le
  t\}}.\end{equation*}
 As a consequence, if moreover  $\tau^{dq}$ satisfies the Jacod's equivalence hypothesis,
  then the $\mathbb{F}$-conditional law of $\tau$, besides an absolutely continuous part which has a density,
   contains a discontinuous part jumping on $\mathbb{F}$-predictable stopping times.~In particular, when
   the enveloping sequence of $\tau^{dp}$ is finite, $\tau$  is under
   the \textit{generalized density hypothesis}
 introduced  in \cite{jiao-li15}.
\end{remark}
\subsection{Enlargement by a general filtration}
We now extend Theorem \ref{prop-repres2} to the more general setting of Section 3 and 4, that is to the case where $\mathbb{H}$
is the reference filtration of
a square integrable martingale $N$, which enjoys the $\mathbb{H}$-PRP.~We recall that conditions \textrm{(D)} and \textrm{(A1)} are in force and the
decomposition (\ref{Yoeurp-decomposition}) holds.~We indicate with
  $\{\tau^{dp}_l\}_{l\in \mathbb{N}}$ the jump
times of $N^{dp}$ and, for each $l\in \mathbb{N}$,
with $\{\tau_{l,h}\}_{h\in\mathbb{ N}}$ an enveloping sequence of $\tau^{dp}_l$.~As before $L=(L_t)_{t\in [0,T]}$ denotes
the derivative process  $dP/dP^*\big|_{\mathcal{G}_t}, t\in [0,T]$.
\begin{lemma}\label{introduction of N prime} Let $L$ be an element of $\mathcal{M}^2_{loc}(P^*,\mathbb{G})$.~Then
\begin{itemize}
\item [(i)] $N^{c}, N^{dq}$ and $N^{dp}$ are special $(P,\mathbb{G})$-semi-martingales;
\item [(ii)] if $N^{dp}$ is of integrable variation under $P$, then the local martingale part $N^\prime$ of the
special $(P,\mathbb{G})$-semi-martingale $N$ satisfies
\begin{equation}\label{def-N prime}
N^\prime=N-A^{N^{c},P, \mathbb{G}}-A^{N^{dq},P, \mathbb{G}}-\sum_{l,h\in \mathbb{N}}\,E^P\left[\Delta N_{\tau_{l,h}}\;
\mathbb{I}_{\tau^{dp}_l=\tau_{l,h}}
\mid\mathcal{G}_{{\tau_{l,h}}^-}\right]\mathbb{I}_{\tau_{l,h}\le
t},\end{equation}
where $A^{N^{c},P, \mathbb{G}}$ and $A^{N^{dq},P, \mathbb{G}}$ are the predictable finite variation terms of
the special $(P,\mathbb{G})$-semi-martingales $N^{c}$ and $N^{dq}$, respectively.
\end{itemize}
\end{lemma}
\begin{proof}
(i) $N^{c}, N^{dq}$ and $N^{dp}$ are $(P^*,\mathbb{G})$-square-integrable martingales so that the hypothesis on $L$ assures that $\langle N^c,L\rangle^{P^*,\mathbb{G}}$, $\langle N^{dq},L\rangle^{P^*,\mathbb{G}}$ and $\langle N^{dp},L\rangle^{P^*,\mathbb{G}}$ exist (see VII.39 and subsequent discussion, page 227, in \cite{del-me-b}).~The statement follows by Remark 2.1 in \cite{caltor18}.\\
(ii)  $N$ is a special $(P,\mathbb{G})$-semi-martingale as sum of special $(P,\mathbb{G})$-semi-martingales.~Moreover, by the uniqueness of
the canonical decomposition the $(P,\mathbb{G})$-predictable finite variation part of $N$ coincides with the sum of the predictable
finite variation parts of $N^{c}, N^{dq}$ and $N^{dp}$.~In particular, since $N^{dp}$ is of integrable variation under $P$, then last
term  can be expressed applying Lemma \ref{prop-explicit-compensator}, so that equality (\ref{def-N prime}) is proved.
\end{proof}
Hypothesis \textrm{(A2)} in this more general framework reads as follows.
 \vspace{3mm}\\
$\mathrm{(A2_N)}$ $P$ is the minimal martingale measure for $N^\prime$ under $P^*$.\vspace{3mm}\\
One  immediately gets the generalization of Proposition \ref{prop: condition A2)}.
\begin{proposition}\label{prop: condition A2_N)}
\phantom{bla}
\begin{itemize}
  \item [(i)] $\mathrm{(A2_N)}$ holds if and only if $M$ and $[M,N]$ are $(P,\mathbb{G})$-local martingales.
    \item [(ii)]  $\mathrm{(A2_N)}$ implies
$\mathbb{F}\underset{P}{\hookrightarrow} \mathbb{G}$.
      \item [(iii)] When $[M,N]\equiv 0$, $\mathrm{(A2_N)}$ is equivalent to $\mathbb{F}\underset{P}{\hookrightarrow} \mathbb{G}$.
\end{itemize}
\end{proposition}
Let us now present the main theorem.
\begin{theorem}\label{prop-repres3}
Assume that \textrm{(D)}, \textrm{(A1)} and $\mathrm{(A2_N)}$ hold, $L$ belongs to $\mathcal{M}^2_{loc}(P^*,\mathbb{G})$ and the accessible martingale part of $N$, $N^{dp}$, is of integrable variation.~Then
the triplet $(M,N^\prime,[M,N])$ enjoys the $(P,\mathbb{G})$-PRP, where $N^\prime$ can be represented as in (\ref{def-N prime}).
\end{theorem}
\begin{proof}
$(M,N,[M,N])$ enjoys the $(P^*,\mathbb{G})$-PRP, that is the analogous of point (i) of Proposition \ref{prop-repres1} holds.~The assumption $\mathrm{(A2_N)}$ implies that $\langle L,
M\rangle^{P^*,\mathbb{G}}=\langle L, N\rangle^{P^*,\mathbb{G}}\equiv 0$ (see point i) of Proposition \ref{prop: condition A2_N)})
and $\langle L,
    [M,N]\rangle^{P^*,\mathbb{G}}$ exist by the regularity hypothesis on $L$.~Then Proposition \ref{lemma:invariance} can be applied, so that, following the same steps of the proof of Theorem \ref{prop-repres2}, we derive the thesis.
\end{proof}
Under the same assumptions of previous theorem following corollary holds.
\begin{corollary}
If $[M,N]\equiv 0$, then  the pair $(M,N^\prime)$ is a $(P,\mathbb{G})$-basis.
\end{corollary}
\begin{proof}
By previous theorem $(M,N^\prime)$ enjoys the $(P, \mathbb{G})$-PRP.~By
Proposition \ref{prop-regularity}  the local martingale part $N^\prime$ of the special $(P, \mathbb{G})$-semi-martingale
$N$ is square-integrable.~Moreover following the same argument of previous section (see (\ref{eq-equivalence-i-iii})) we derive  $M$ and $N^\prime$ to be orthogonal.
\end{proof}
\begin{example}
An example of martingale representation on progressive enlargement by an accessible
random time is studied in \cite{caltor-proc19}.~Here we propose a slight generalization
in order to get an example
of martingales representation on a filtration enlarged through a full process.\\
On $(\Omega,\mathcal{F},P)$ let $W$ be a standard Brownian motion and $H^\eta$ the compensated occurrence
process of a random time $\eta$ with values in the set
$\{1,2,3\}$.~Assume $W$
and $\eta$ independent.~Applying point i) of Proposition \ref{lemma: M+N basis} to $W$ and $H^\eta$ one gets condition \textrm{(A1)}
for $M$ and $\mathbb{F}$ defined by
$$M:=W+H^\eta,\;\;\;\;\mathbb{F}:=\mathbb{F}^W\vee\sigma(\eta\wedge \cdot),$$
where $\mathbb{F}^W$ denotes the natural filtration of $W$.\\
On the same probability space let $H$ be the compensated occurrence
process of a random time $\tau$ taking values in
the set $\{2,4\}$ and let  $\Pi$ be a compensated Poisson martingale independent of $(W,\eta,\tau)$ with natural filtration $\mathbb{F}^\Pi$.~Set
$$N:=H+\Pi,\;\;\;\;\mathbb{H}:=\sigma(\tau\wedge \cdot)\vee\mathbb{F}^\Pi.$$
 $N$ is a $(P,\mathbb{H})$-martingale and by point (i) of Proposition \ref{lemma: M+N basis} $N$ enjoys the $(P,\mathbb{H})$-PRP.\\
If $W$ is independent of $(\eta, \tau)$ and the
 joint law
$p_{\eta, \tau}$ of $(\eta, \tau)$ is strictly positive on the set
$\{1,2,3\}\times\{2,4\}$, then condition \textrm{(D)} holds and the martingale
preserving measure $P^*$ is defined on
$(\Omega,\mathcal{G}_T)$, up to a standard extension procedure, by
the rule $$P^*(A\cap C\cap D\cap E):=P(A)P(C)P(D),$$ for any $A\in\mathcal{F}_T^W, C\in\sigma(\eta\wedge T), D\in\sigma(\tau\wedge T), E\in\mathcal{F}^\Pi_T$.~In this case $L$ is a bounded process.
\\$[M,\Pi]\equiv 0$ so that $[M,N]\equiv [M,H]$.~Therefore
condition $\mathrm{(A2_N)}$ is in force as soon as $[M,H]$ is a  $(P,\mathbb{G})$-martingale.~It can be easily derived that this happens if
\begin{equation}\label{computations}
P\left(\tau=2\mid\eta= 2\right)=P(\tau=2\mid\eta\neq
1)=P(\tau=2\mid\eta=3).
\end{equation}
So, since $\Pi$ is a $(P,\mathbb{G})$-martingale then $N^\prime$ coincides with $H^\prime + \Pi$ and
the above theorem states the $(P,\mathbb{G})$-driving martingale is
the triplet
$\Big(M,H^\prime + \Pi, [M,H]\Big)$.~We refer to \cite{caltor-proc19} for the computation of the conditions (\ref{computations})
 and for the explicit expression of $H^\prime$ and $[M,H]$.
 \end{example}
\section{Conclusions and perspectives}
Our paper has been mainly inspired by the literature of credit risk theory, where the problem of martingale representation on a market with full information has been largely studied.~Actually in real markets the default time may be predictable, may overlap the jump of the asset price process or even the insider may be aware of more than a single default.~That is why, working in a purely theoretical framework, we have assumed that the filtration enlargement is induced by a not necessarily quasi-left continuous random process, possibly with jump times common to the processes adapted to the basic filtration and more general that the occurrence process of a random time.~As a byproduct we have investigated
the interplay between two notions specific of mathematical finance: the \textit{minimal martingale measure} and the \textit{immersion property}.
\bigskip\\
A natural development of the current result, which is object of ongoing research, is to get rid of the decoupling assumption \textrm{(D)} by moving on
the study of martingale representation on enlarged filtrations
in the light of the recent paper
by Aksamit, Choully and Jeanblanc on \textit{thin-thick decomposition of a random time} (see \cite{ak-chou-jean-18}).
\section*{Funding}
\noindent Partially supported by  the MIUR Excellence Department Project MatMod@TOV awarded to the Department of Mathematics, University of Rome Tor Vergata, CUP E83C23000330006.

\bibliographystyle{acm}
\bibliography{biblio}
\end{document}